\documentclass[oneside,english,reqno]{amsart}
\usepackage[T1]{fontenc}
\usepackage[latin9]{inputenc}
\usepackage{babel}
\usepackage{amsthm}
\usepackage{amssymb}
\usepackage{graphicx}
\usepackage[unicode=true,pdfusetitle,
 bookmarks=true,bookmarksnumbered=false,bookmarksopen=false,
 breaklinks=false,pdfborder={0 0 1},backref=false,colorlinks=false]
 {hyperref}
\usepackage{breakurl}

\makeatletter

\providecommand{\tabularnewline}{\\}

\theoremstyle{plain}
\newtheorem{thm}{\protect\theoremname}[section]
  \theoremstyle{definition}
  \newtheorem{defn}[thm]{\protect\definitionname}
  \theoremstyle{plain}
  \newtheorem{prop}[thm]{\protect\propositionname}
  \theoremstyle{plain}
  \newtheorem{lem}[thm]{\protect\lemmaname}
  \theoremstyle{definition}
  \newtheorem{example}[thm]{\protect\examplename}
  \theoremstyle{plain}
  \newtheorem{algorithm}[thm]{\protect\algorithmname}
  \theoremstyle{plain}
  \newtheorem{assumption}[thm]{\protect\assumptionname}
  \theoremstyle{remark}
  \newtheorem{rem}[thm]{\protect\remarkname}
  \theoremstyle{remark}
  \newtheorem*{acknowledgement*}{\protect\acknowledgementname}

\newcommand{\lev}{\mbox{\rm lev}}

\newcommand{\diam}{\mbox{\rm diam}}

\numberwithin{equation}{section}
\numberwithin{figure}{section}
\title[Principles for mountain pass algorithms, and the parallel distance]{Some principles for mountain pass algorithms, and the parallel distance}

\makeatother

  \providecommand{\acknowledgementname}{Acknowledgement}
  \providecommand{\algorithmname}{Algorithm}
  \providecommand{\assumptionname}{Assumption}
  \providecommand{\definitionname}{Definition}
  \providecommand{\examplename}{Example}
  \providecommand{\lemmaname}{Lemma}
  \providecommand{\propositionname}{Proposition}
  \providecommand{\remarkname}{Remark}
\providecommand{\theoremname}{Theorem}

\begin{document}

\date{\today{}}

\author{Justin T. Brereton }

\address{Department of Mathematics, Massachusetts Institute of Technology,
Cambridge, MA, USA}

\email{jbrere@mit.edu}

\author{C.H. Jeffrey Pang}

\address{Department of Mathematics, National University of Singapore, Block
S17 05-10, 10 Lower Kent Ridge Road, Singapore 119076 }

\email{matpchj@nus.edu.sg}
\begin{abstract}
The problem of computing saddle points is important in certain problems
in numerical partial differential equations and computational chemistry,
and is often solved numerically by a minimization problem over a set
of mountain passes. We point out that a good global mountain pass
algorithm should have good local and global properties. Next, we define
the parallel distance, and show that the square of the parallel distance
has a quadratic property. We show how to design algorithms for the
mountain pass problem based on perturbing parameters of the parallel
distance, and that methods based on the parallel distance have midrange
local and global properties.
\end{abstract}
\maketitle
\tableofcontents{}

\section{Introduction}

We begin with the definition of a mountain pass.
\begin{defn}
(Mountain pass) \label{def:mountain-pass}Let $X$ be a topological
space, and consider $a,b\in X$. Let $\Gamma(a,b)$ be the set of
continuous paths $p:[0,1]\rightarrow X$ such that $p(0)=a$ and $p(1)=b$.
For a function $f:X\rightarrow\mathbb{R}$, define an \emph{optimal
mountain pass} $\bar{p}\in\Gamma(a,b)$ to be a minimizer of the problem
\begin{equation}
\inf_{p\in\Gamma(a,b)}\sup_{0\leq t\leq1}f\circ p(t).\label{eq:mtn-pass}
\end{equation}

The point $\bar{x}$ is a \emph{critical point }if $\nabla f(\bar{x})=0$,
and the critical point $\bar{x}$ is a \emph{saddle point }if it is
not a local maximizer or minimizer on $X$. The value $f(x)$ is a\emph{
critical value }if $x$ is a critical point. We say that $\bar{x}$
is a \emph{saddle point of mountain pass type }if there is an open
set $U$ containing $\bar{x}$ such that $\bar{x}$ lies in the closure
of two path connected components of $\{x\in U:f(x)<f(\bar{x})\}$.
In the case where $f$ is smooth and an optimal mountain pass $\bar{p}:[0,1]\to X$
exists, the maximum of $f$ on $\bar{p}([0,1])$ is a\emph{ }saddle
point.
\end{defn}
In this paper, we shall focus on the case where $X=\mathbb{R}^{n}$
and the saddle point is nondegenerate. A saddle point $\bar{x}$ is
said to be \emph{nondegenerate} if $\nabla^{2}f(\bar{x})$ is invertible.
Moreover, a nondegenerate saddle point $\bar{x}$ has \emph{Morse
index one }if $\nabla^{2}f(\bar{x})$ contains exactly one negative
eigenvalue.

The problem of finding saddle points numerically is important in
the problem of finding weak solutions to partial differential equations
numerically. The first critical point existence theorems now known
as the mountain pass theorems were proved in \cite{AR73,Rab77}. Some
recent theoretical references include \cite{MW89,Rabino86,Schechter99,Struwe08,Willem96}.
See also the more accessible reference \cite{Jabri03}. The original
paper of a mountain pass algorithm to solve partial differential equations
is \cite{CM93}, and it contains several semilinear elliptic problems.
Particular applications in numerical partial differential equations
include finding periodic solutions of a boundary value problem modeling
a suspension bridge \cite{Feng94} (introduced by \cite{LM91}), studying
a system of Ginzburg-Landau type equations arising in the thin film
model of superconductivity \cite{GM08}, the choreographical 3-body
problem \cite{ABT06}, and cylinder buckling \cite{HLP06}. Other
notable works in computing saddle points for solving numerical partial
differential equations include the use of constrained optimization
\cite{Hor04}, extending the mountain pass algorithm to find saddle
points of higher Morse index \cite{DCC99,LZ01}, extending the mountain
pass algorithm to find nonsmooth saddle points \cite{YZ05}, and using
symmetry \cite{WZ04,WZ05}.

The problem of finding saddle points numerically is by now well entrenched
in the chemistry curriculum. In transition state theory, the problem
of finding the least amount of energy to transition between two stable
states is equivalent to finding an optimal mountain pass between these
two stable states. The highest point on the optimal mountain pass
can then be used to determine the reaction kinetics. The foundations
of transition state theory was laid by Marcelin, and important work
by Eyring and Polanyi in 1931 and by Pelzer and Wigner a year later
established the importance of saddle points in transition state theory.
We cite the Wikipedia entry on transition state theory for more on
its history and further references. Numerous methods for computing
saddle points were suggested through the years, and we refer to the
surveys \cite{HJJ00,Schlegel05,Schlegel11,Wales06} as well as the
recent text \cite{Wales03}. A software for computing saddle points
in chemistry is Gaussian%
\footnote{\href{http://www.gaussian.com/}{http://www.gaussian.com/}%
}. Tools for computing transition states%
\footnote{\href{http://theory.cm.utexas.edu/vtsttools/neb/}{http://theory.cm.utexas.edu/vtsttools/neb/}%
} are also included in VASP%
\footnote{\href{http://cms.mpi.univie.ac.at/vasp/vasp/vasp.html}{http://cms.mpi.univie.ac.at/vasp/vasp/vasp.html}%
}. Though the entire optimal mountain pass is needed for such an application,
the process of computing saddle points often gives hints on an optimal
mountain pass. 

As mentioned in \cite{mountain}, our initial interest in the problem
of computing saddle points of mountain pass type comes from computing
the distance of a matrix $A\in\mathbb{C}^{n\times n}$ to the closest
matrix with repeated eigenvalues (also known as the Wilkinson distance
problem). 

\begin{figure}[h]
\begin{tabular}{|c|c|}
\hline 
\includegraphics[scale=0.4]{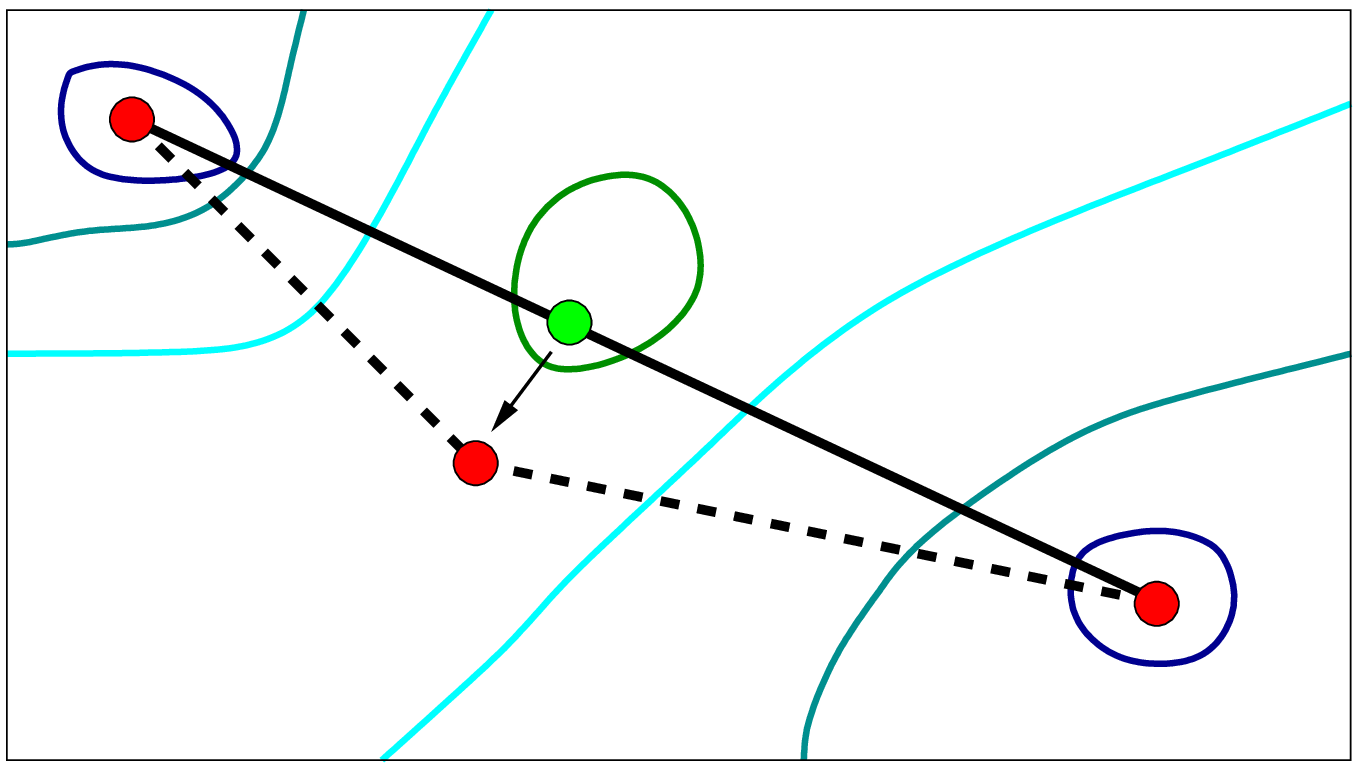} & \includegraphics[scale=0.3]{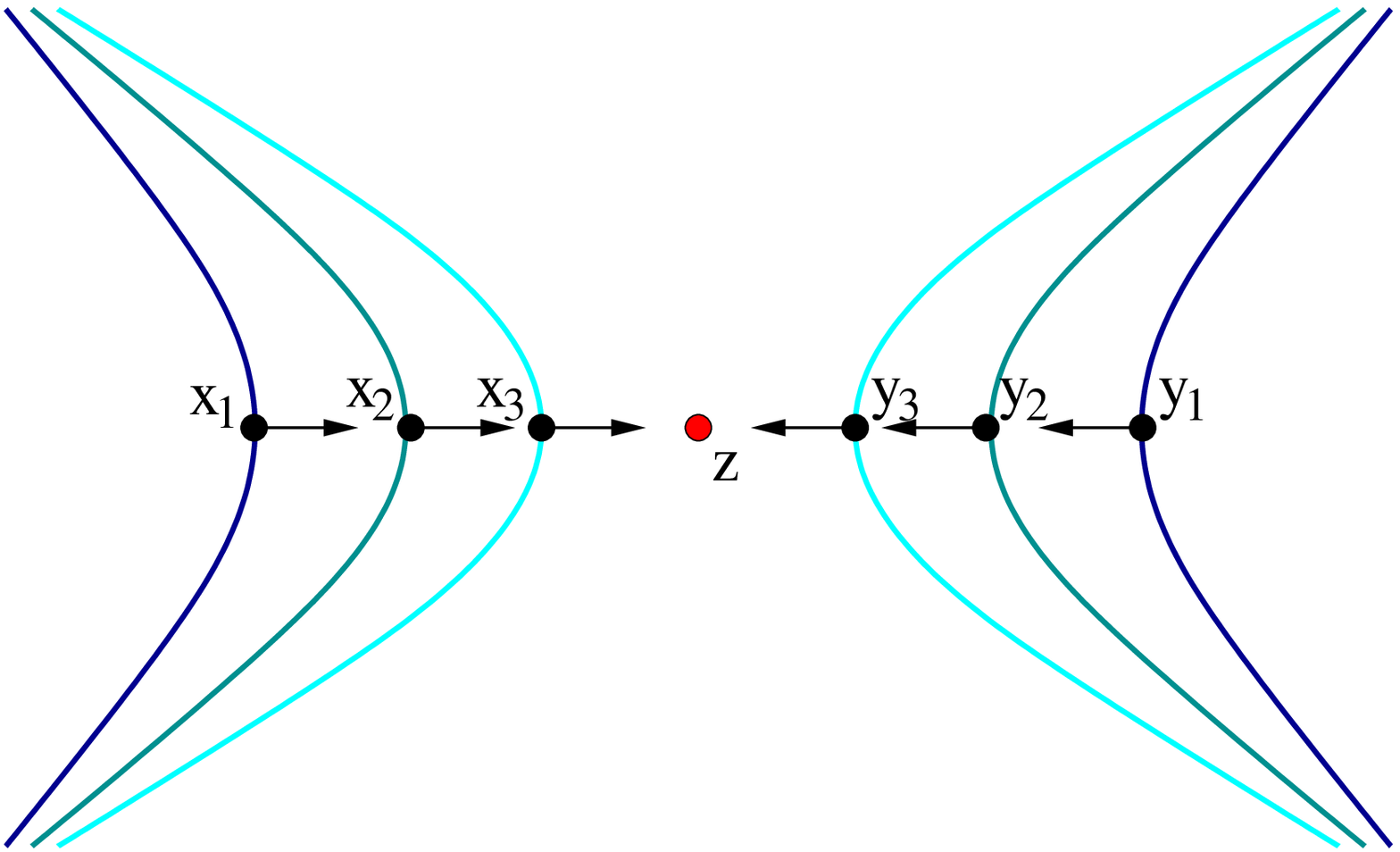}\tabularnewline
\hline 
\end{tabular}

\caption{\label{fig:mtn-contrast}The diagram on the left illustrates a path-based
method, while the diagram on the right illustrates a level set method,
as was done in \cite{mountain} and this paper.}
\end{figure}

We recall three broad methods for computing the mountain pass:

\subsection*{Path-based methods}

The typical mountain pass algorithm makes use of the formula in \eqref{eq:mtn-pass}
to find a saddle point. The paths in $\Gamma(a,b)$ are discretized,
and perturbed so that the maximum value of $f$ along the path is
reduced. The point on an optimizing path attaining the maximum value
is a good estimate of the critical point. See Figure \ref{fig:mtn-contrast}.

\subsection*{Quadratic model methods}

Once the iterates are close enough to the saddle point $\bar{x}$,
the quadratic expansion 
\begin{equation}
f(x)=\frac{1}{2}(x-\bar{x})^{T}\nabla^{2}f(\bar{x})(x-\bar{x})+\nabla f(\bar{x})^{T}(x-\bar{x})+f(\bar{x})+o(\|x-\bar{x}\|^{2})\label{eq:quad-approx}
\end{equation}
can form the basis of algorithms that converge quickly to the saddle
point. A Newton method can achieve quadratic convergence to the saddle
point, or its variants can achieve fast convergence. The gradient
$\nabla f(x)$ has close to linear behavior, and other methods involving
solving the linear system are also possible.

\subsection*{Level set methods}

In \cite{mountain,MironFichthorn}, a different strategy of using
\emph{level sets} 
\[
\lev_{\leq l}f:=\{x:f(x)\leq l\}
\]
 is suggested: For a neighborhood $U$ of the critical point $\bar{x}$
and an increasing sequence of $l_{i}$ converging to the critical
value $f(\bar{x})$, find the closest points in different components
of $U\cap\lev_{\leq l_{i}}f$, say $x_{i}$ and $y_{i}$. Figure \ref{fig:mtn-contrast}
contrasts path-based methods and level set methods. Under additional
conditions, $\{x_{i}\}_{i=1}^{\infty}$ and $\{y_{i}\}_{i=1}^{\infty}$
both converge to $\bar{x}$. An optimal mountain pass can be estimated
from the iterates $\{x_{i}\}_{i=1}^{\infty}$ and $\{y_{i}\}_{i=1}^{\infty}$.
Advantages of level set methods over path-based methods include: 
\begin{itemize}
\item [(A1)]The level set method needs only to keep track of two points
at each step instead of an entire path.
\item [(A2)]The bulk of computations are performed near the saddle point. 
\item [(A3)]The distance between the components of the level set indicate
the performance of the algorithm.
\item [(A4)]Provided black boxes for finding closest points to components
of the level set and for the minimization of the function $f$ on
an affine space exist, an algorithm locally superlinearly convergent
to the critical point is described in \cite{mountain}. See also (D1)
in Section \ref{sec:global-framework}.
\end{itemize}
However, here are some difficulties encountered in the level set algorithm
in \cite{mountain}, which we will elaborate in Section \ref{sec:global-framework}.

One contribution we make in this paper is to identify properties desirable
for a global mountain pass algorithm. Specifically, we propose these
two principles:
\begin{enumerate}
\item [(P1)]Suppose $f\in\mathcal{C}^{2}$. Once the iterates are close
enough to a nondegenerate saddle point of Morse index one, the algorithm
should converge quickly to the saddle point $\bar{x}$.
\item [(P2)]The global algorithm should find a saddle point of mountain
pass type.
\end{enumerate}
The analogy to Principle (P2) in optimization is to seek decrease
so that iterates converge to a local minimizer. Principle (P1) states
that the algorithm should have fast convergence once close enough
to a saddle point. Related to Principle (P1) is Principle (P1$^{\prime}$)
below.
\begin{enumerate}
\item [(P1$^{\prime}$)]For the quadratic $f(x)=\frac{1}{2}x^{T}Hx+g^{T}x+c$,
where $H$ is an invertible symmetric matrix with one negative eigenvalue
and $n-1$ positive eigenvalues, the algorithm should have excellent
convergence.
\end{enumerate}
We make a short summary of the performance of the various mountain
pass algorithms. Path-based methods excel in (P2) due to the proof
of the mountain pass theorem of \cite{AR73} using the Ekeland variational
principle. More specifically, under suitable conditions, if $p_{i}(\cdot)$
is a sequence of paths in $\Gamma(a,b)$ such that $\max_{t\in[0,1]}f\circ p_{i}$
converges to the critical level, then the sequence of maximizers of
$f$ along the path $p_{i}(\cdot)$ converge to a saddle point. However,
it does poorly for (P1) and (P1$^{\prime}$) because it does not take
advantage of the quadratic approximation \eqref{eq:quad-approx} to
achieve fast convergence. On the other hand, methods that make extensive
use of the quadratic approximation \eqref{eq:quad-approx} excel for
(P1) and (P1$^{\prime}$), but does not satisfy (P2) because the quadratic
approximation need not be valid globally. 

Another contribution of this paper is to argue that level set methods
should be part of a good mountain pass algorithm because it does well
for the Principles (P1), (P1$^{\prime}$) and (P2).

 We also show how the parallel distance defined below can be
part of a good mountain pass algorithm. For a set $C\subset\mathbb{R}^{n}$,
its \emph{diameter} $\diam(C)$ is defined by $\diam(C):=\sup\{|x-y|:x,y\in C\}$.
\begin{defn}
\label{def:glv}(Parallel distance) Let $f:\mathbb{R}^{n}\to\mathbb{R}$
be $\mathcal{C}^{2}$ in a convex neighborhood $U^{\prime}$, and
let $v$ be a unit vector. See Figure \ref{fig:glv}. Consider the
set $S_{l,v}(x)\subset\mathbb{R}^{n}$ defined by 
\[
S_{l,v}(x):=U^{\prime}\cap[\{x\}+\mathbb{R}\{v\}]\cap\lev_{\geq l}f.
\]
For a neighborhood $U$ of $\bar{x}$ such that $U\subset U^{\prime}$,
define the \emph{parallel distance }$g_{l,v}:U\to\mathbb{R}$ by 
\[
g_{l,v}(x):=\diam\big(S_{l,v}(x)\big).
\]
When $S_{l,v}(x)=\emptyset$, $g_{l,v}(x)=0$. In the case where $S_{l,v}(x)$
is a line segment, we can write $g_{l,v}(x)$ as 
\begin{equation}
g_{l,v}(x)=g_{l,v,1}(x)+g_{l,v,2}(x),\label{eq:g-diff-formula}
\end{equation}
where\begin{subequations} 
\begin{eqnarray}
g_{l,v,1}(x) & = & \max\{\phantom{-}v^{T}z^{\phantom{\prime}}\mid f(z^{\phantom{\prime}})=l,\, z^{\phantom{\prime}}\in S_{l,v}(x)\},\label{eq:glv1}\\
\mbox{ and }g_{l,v,2}(x) & = & \max\{-v^{T}z^{\prime}\mid f(z^{\prime})=l,\, z^{\prime}\in S_{l,v}(x)\}.\label{eq:glv2}
\end{eqnarray}
\end{subequations}Also, define $z(x)$ and $z^{\prime}(x)$ as
\begin{eqnarray*}
z^{\phantom{\prime}}(x) & = & \arg\max\{\phantom{-}v^{T}z^{\phantom{\prime}}\mid f(z^{\phantom{\prime}})=l,\, z^{\phantom{\prime}}\in S_{l,v}(x)\},\\
\mbox{ and }z^{\prime}(x) & = & \arg\max\{-v^{T}z^{\prime}\mid f(z^{\prime})=l,\, z^{\prime}\in S_{l,v}(x)\}.
\end{eqnarray*}

\end{defn}
\begin{figure}
\includegraphics[scale=0.4]{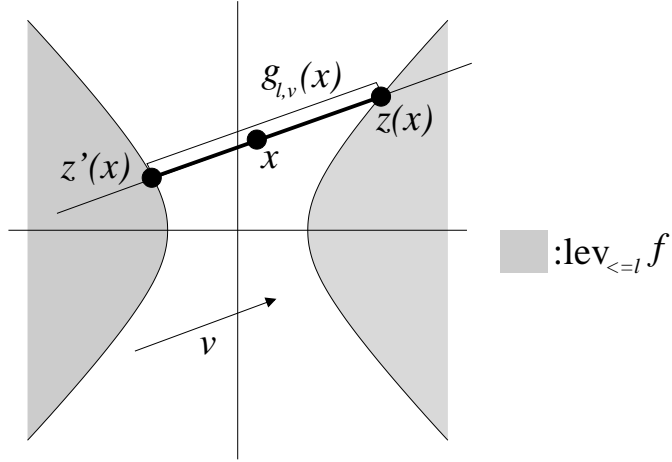}\caption{\label{fig:glv}Illustration of $g_{l,v}(x)$ in Definition \ref{def:glv}. }
\end{figure}

One step of the mountain pass algorithm in \cite{mountain} is to
find the closest points between components of the level sets. The
problem of finding the closest points between two sets is not necessarily
easy, and an alternating projection algorithm converges slowly once
close to the optimum points. We will show that as long as $v$ is
close enough to the eigenvector corresponding to the negative eigenvalue
of the Hessian of the saddle point, the square of the parallel distance
satisfies property (P1). This allows us to get around the problem
of finding the closest points between components of the level sets.

\subsection{Outline of paper}

Section \ref{sec:Basic-parallel} discusses various basic properties
of the parallel distance. The topics discussed are: how the square
of the parallel distance satisfies (P1$^{\prime}$), formulas for
the gradient and Hessian of the parallel distance $g_{l,v}(\cdot)$
and its square $g_{l,v}(\cdot)^{2}$, and why it is preferable to
consider $g_{l,v}(\cdot)^{2}$ for the smooth problem instead of $g_{l,v}(\cdot)$.
Section \ref{sec:global-framework} proposes subroutines for a mountain
pass algorithm, and discusses how to use these subroutines to design
a mountain pass algorithm with midrange local and global properies.
Section \ref{sec:Convexity-para-dist} shows that the Hessian $\nabla^{2}(g_{l,v}^{2})(\cdot)$
is close to the Hessian as predicted by a quadratic model. This shows
that the Hessian $\nabla^{2}(g_{l,v}^{2})(\cdot)$ is not sensitive
to $l$ as the computations get close to the saddle point, making
old estimates of $\nabla^{2}(g_{l,v}^{2})(\cdot)$ useful for future
computations involving a different $l$. Section \ref{sec:implement}
shows how our algorithm performs in an implementation.

\section{\label{sec:Basic-parallel}Basic properties of the parallel distance}

In this section, we study basic properties of the parallel distance
function.

When $f$ is an exact quadratic whose critical point is nondegenerate
of Morse index one, we have the following appealing result. 
\begin{prop}
\label{pro:quad-formula}(Quadratic formula for square of parallel
distance in exact quadratic) Suppose that $f:\mathbb{R}^{n}\to\mathbb{R}$
is an exact quadratic $f(x)=\frac{1}{2}x^{T}Hx+g^{T}x+c$, with $H\in\mathbb{R}^{n\times n}$
having $n-1$ positive eigenvalues and one negative eigenvalue. Consider
a unit vector $v$ such that $v^{T}Hv<0$. Then $S_{l,v}(x)$ is a
line segment, and the function $g_{l,v}(\cdot)$ takes the form \eqref{eq:g-diff-formula}.
Additionally, we have 
\begin{align}
g_{l,v}(x)^{2} & =\max\Big\{0,\frac{4}{(v^{T}Hv)^{2}}\Big[x^{T}[Hvv^{T}H-(v^{T}Hv)H]x\nonumber \\
 & \qquad\qquad\qquad\qquad\qquad+2[(g^{T}v)v^{T}H-(v^{T}Hv)g^{T}]x\nonumber \\
 & \qquad\qquad\qquad\qquad\qquad+\big[(g^{T}v)^{2}+(v^{T}Hv)[-2c+2l]\big]\Big]\Big\}.\label{eq:g-l-v-full}
\end{align}
For $g_{l,v}(x)>0$ and $v$ sufficiently close to the eigenvector
corresponding to the negative eigenvalue of \textup{$H$}, the matrix
$\nabla^{2}(g_{l,v}^{2})(x)$ has $n-1$ positive eigenvalues and
one zero eigenvalue. The function $g_{l,v}^{2}$ is convex. Moreover,
let $\bar{x}$ be the saddle point $-H^{-1}g$. If $l<f(\bar{x})$
and $g_{l,v}^{2}$ has a minimizer $\tilde{x}$, then $\bar{x}$ is
the midpoint of the intersection of the line $\{\tilde{x}\}+\mathbb{R}\{v\}$
and $\lev_{=l}f$.  \end{prop}
\begin{proof}
For the case of the quadratic $f$, the neighborhoods $U$ and $U^{\prime}$
can be taken to be $\mathbb{R}^{n}$.  The value $g_{l,v}(x)$ can
be computed as follows. At where $g_{l,v}(x)>0$, let $x+t_{i}v$,
where $t_{i}\in\mathbb{R}$ and $i=1,2$, be two points of intersection
of the line $\{x\}+\mathbb{R}\{v\}$ and the curve $\lev_{=l}f$.
The $t_{i}$'s can be calculated as follows:

\begin{eqnarray*}
\frac{1}{2}(x+t_{i}v)^{T}H(x+t_{i}v)+g^{T}(x+t_{i}v)+c & = & l\\
\Rightarrow(v^{T}Hv)t_{i}^{2}+[2(v^{T}Hx)+2g^{T}v]t_{i}+x^{T}Hx+2g^{T}x+2c-2l & = & 0.
\end{eqnarray*}
We have 
\[
t_{i}=\frac{-2(v^{T}Hx)-2g^{T}v\pm\sqrt{4(v^{T}Hx+g^{T}v)^{2}-4(v^{T}Hv)(x^{T}Hx+2g^{T}x+2c-2l)}}{2(v^{T}Hv)}.
\]
This gives 
\begin{eqnarray*}
g_{l,v}(x) & = & \frac{2}{v^{T}Hv}\sqrt{(v^{T}Hx+g^{T}v)^{2}-(v^{T}Hv)(x^{T}Hx+2g^{T}x+2c-2l)}\\
g_{l,v}(x)^{2} & = & \frac{4}{(v^{T}Hv)^{2}}[(v^{T}Hx+g^{T}v)^{2}-(v^{T}Hv)(x^{T}Hx+2g^{T}x+2c-2l)]\\
 & = & \frac{4}{(v^{T}Hv)^{2}}\Big[x^{T}[Hvv^{T}H-(v^{T}Hv)H]x+2[(g^{T}v)v^{T}H-(v^{T}Hv)g^{T}]x\\
 &  & \qquad\qquad\qquad+\big[(g^{T}v)^{2}+(v^{T}Hv)[-2c+2l]\big]\Big].
\end{eqnarray*}
Taking into account the fact that $g_{l,v}(x)$ can equal zero, $g_{l,v}(x)^{2}$
has the formula as given in \eqref{eq:g-l-v-full}. For the case when
$v=\bar{v}$, the eigenvector corresponding to the negative eigenvalue
of $H$, we find that $\bar{\ensuremath{v}}$ is the eigenvector corresponding
the zero eigenvalue for the Hessian 
\[
\nabla^{2}(g_{l,\bar{v}}^{2})(x)=\frac{8}{(v^{T}Hv)^{2}}[Hvv^{T}H-(v^{T}Hv)H].
\]
The other eigenvalues of $\nabla^{2}(g_{l,\bar{v}}^{2})(x)$ can easily
be calculated to be $-8\lambda_{i}/\lambda_{n}$ for $i=1,\dots,n-1$,
where $\lambda_{i}$s are the eigenvalues of $H$ arranged in decreasing
order. 

Note that $v$ is an eigenvector corresponding to eigenvalue zero
of $\nabla^{2}(g_{l,v}^{2})(x)$. Recall that the eigenvalues depend
continuously on the matrix entries. If the unit vector $v$ is sufficiently
close to $\bar{v}$, then the Hessian $\nabla^{2}(g_{l,v}^{2})(x)$
has one zero eigenvalue and $n-1$ positive eigenvalues. The convexity
of $g_{l,v}^{2}(\cdot)$ is clear. 

In the case where $l<f(\bar{x})$, it is easy to check that $x=-H^{-1}g$
is a minimizer of $g_{l,v}^{2}(\cdot)$. The other claims are easy.
\end{proof}
Proposition \ref{pro:quad-formula} says that when $f$ is quadratic,
then $g_{l,v}(\cdot)^{2}$ is also a quadratic, so a mountain pass
algorithm based on the parallel distance will satisfy (P1$^{\prime}$).

We next show that the parallel distance behaves well near the saddle
point $\bar{x}$ of Morse index one.
\begin{prop}
\label{pro:Well-defined-g}(Behavior near saddle point) Let $f:\mathbb{R}^{n}\to\mathbb{R}$
be $\mathcal{C}^{2}$ in a neighborhood of a nondegenerate saddle
point $\bar{x}$ of Morse index one, and $\bar{v}$ be the eigenvector
of unit length corresponding to the negative eigenvector of $\nabla^{2}f(\bar{x})$.
For $\delta>0$, define $\bar{f}_{\delta}:\mathbb{R}^{n}\to\mathbb{R}$
and $\bar{S}_{\delta,l,v}(x)$ by 
\begin{eqnarray}
\bar{f}_{\delta}(x) & := & \frac{1}{2}(x-\bar{x})^{T}[\nabla^{2}f(\bar{x})+\delta I](x-\bar{x})+f(\bar{x}),\label{eq:bar-f}\\
\mbox{ and }\bar{S}_{\delta,l,v}(x) & := & [\{x\}+\mathbb{R}\{v\}]\cap\lev_{\geq l}\bar{f}_{\delta}.\nonumber 
\end{eqnarray}
There is a neighborhood $U^{\prime}$ of $\bar{x}$ and $\epsilon>0$
such that:
\begin{enumerate}
\item $\|\nabla^{2}f(\bar{x})-\nabla^{2}f(x)\|<\delta$ for all $x\in U^{\prime}$. 
\item If $\|v-\bar{v}\|<\epsilon$, then the map $t\mapsto f(x+tv)$, where
$t\in\mathbb{R}$, is concave at wherever $x+tv\in U^{\prime}$. Hence
$S_{l,v}(x)$ is either a line segment or an empty set.
\item If $v$ is a unit vector satisfying $\|v-\bar{v}\|<\epsilon$ and
$|l-f(\bar{x})|<\epsilon$, then for all $x\in\mathbb{B}_{\epsilon}(\bar{x})$,
we have $S_{l,v}(x)\cap U^{\prime}\subset\bar{S}_{\delta,l,v}(x)\subset U^{\prime}$
 (which includes the case $S_{l,v}(x)=\emptyset$). 
\end{enumerate}
\end{prop}
\begin{proof}
The statement (1) holds for some $U^{\prime}$ of $\bar{x}$. We can
shrink $U^{\prime}$ if necessary so that $\bar{v}^{T}\nabla^{2}f(x)\bar{v}<0$
for all $x\in U^{\prime}$, and an $\epsilon>0$ can be found so that
(2) is satisfied. 

Choose $\gamma>0$ such that $\mathbb{B}_{\gamma}(\bar{x})\subset U^{\prime}$.
Then condition (1) ensures that $f(x)<\bar{f}_{\delta}(x)$ for all
$x\in U^{\prime}$, so $S_{l,v}(x)\cap U^{\prime}\subset\bar{S}_{\delta,l,v}(x)$.
The endpoints of the line segment $\bar{S}_{l,v}(x)$ are of the form
$[x+\tilde{t}_{1}v,x+\tilde{t}_{2}v]$, whose endpoints can be calculated
using the quadratic formula employed in the proof of Proposition \ref{pro:quad-formula}
as $x+\tilde{t}_{i}v$, $i=1,2$, giving us 
\begin{eqnarray*}
\tilde{t}_{i} & = & -\frac{[v^{T}H_{\delta}(x-\bar{x})]}{v^{T}H_{\delta}v}\\
 &  & \pm\frac{\sqrt{4[v^{T}H_{\delta}(x-\bar{x})]^{2}-4[v^{T}H_{\delta}v][(x-\bar{x})^{T}H_{\delta}(x-\bar{x})+2f(\bar{x})-2l]}}{2[v^{T}H_{\delta}v]},
\end{eqnarray*}
where $H_{\delta}=\nabla^{2}f(\bar{x})+\delta I$. The formula above
is continuous in $v$, $x$ and $l$ whenever $\tilde{t}_{i}$ is
real, and as $x\to\bar{x}$ and $l\to f(\bar{x})$, we have $\tilde{t}_{i}\to0$.
From $\|x+\tilde{t}_{i}v-\bar{x}\|\leq\|x-\bar{x}\|+|\tilde{t}_{i}|$,
we can choose $\epsilon$ small enough so that if $\|v-\bar{v}\|<\epsilon$,
$|l-f(\bar{x})|<\epsilon$ and $x\in\mathbb{B}_{\epsilon}(\bar{x})$,
then $\|x+\tilde{t}_{i}v-\bar{x}\|<\gamma$, giving us $\bar{S}_{\delta,l,v}(x)\subset\mathbb{B}_{\gamma}(\bar{x})\subset U^{\prime}$.
This means that condition (3) holds.
\end{proof}
The expression \eqref{eq:g-diff-formula} gives us a way to calculate
derivatives of the parallel distance. We have the following results.
\begin{lem}
\label{lem:Grad-Hess-g}(Gradient and Hessian of $g_{l,v}$) Let $f:\mathbb{R}^{n}\to\mathbb{R}$
be $\mathcal{C}^{2}$ everywhere. Recall the function $g_{l,v}:\mathbb{R}^{n}\to\mathbb{R}$
and the neighborhoods $U$ and $U^{\prime}$ on which $g_{l,v}$ is
defined. Suppose that $g_{l,v}$ can be represented as \eqref{eq:g-diff-formula}.
Let $z(x)$ and $z^{\prime}(x)$ be the respective maximizers in the
definitions of $g_{l,v,1}$ and $g_{l,v,2}$ in \eqref{eq:glv1} and
\eqref{eq:glv2}. Then, provided $\nabla f(z(x))^{T}v\neq0$ and $\nabla f(z^{\prime}(x))^{T}v\neq0$,
we have 
\[
\nabla g_{l,v}(x)=-\frac{\nabla f(z(x))}{\nabla f(z(x))^{T}v}+\frac{\nabla f(z^{\prime}(x))}{\nabla f(z^{\prime}(x))^{T}v}.
\]
To simplify the notation, we suppress the dependence of $z$ and $z^{\prime}$
on $x$. We also have 
\begin{eqnarray*}
\nabla^{2}g_{l,v}(x) & = & -\left(I-\frac{\nabla f(z)v^{T}}{v^{T}\nabla f(z)}\right)\frac{\nabla^{2}f(z)}{v^{T}\nabla f(z)}\left(I-\frac{\nabla f(z)v^{T}}{v^{T}\nabla f(z)}\right)^{T}\\
 &  & +\left(I-\frac{\nabla f(z^{\prime})v^{T}}{v^{T}\nabla f(z^{\prime})}\right)\frac{\nabla^{2}f(z^{\prime})}{v^{T}\nabla f(z^{\prime})}\left(I-\frac{\nabla f(z^{\prime})v^{T}}{v^{T}\nabla f(z^{\prime})}\right)^{T}.
\end{eqnarray*}
\end{lem}
\begin{proof}
Write $F(d,t):=f(x+tv+d)=f\big((x+\bar{t}v)+d+(t-\bar{t})v\big)$.
We evaluate the partial derivatives of $F$ at $(0,\bar{t})$ to be
\[
\nabla_{d}F(0,\bar{t})=\nabla f(x+\bar{t}v)\qquad\mbox{and}\qquad\nabla_{t}F(0,\bar{t})=\nabla f(x+\bar{t}v)^{T}v.
\]
For each $d$, we can find $t$ such that $F(d,t)=0$. By the implicit
function theorem, the derivative of $t$ with respect to $d$ equals
$-\frac{\nabla f(x+\bar{t}v)}{\nabla f(x+\bar{t}v)^{T}v}=-\frac{\nabla f(z(x))}{\nabla f(z(x))^{T}v}$
provided the denominator is nonzero. From this and the fact that $g_{l,v,1}(\cdot)$
and $g_{l,v,2}(\cdot)$ are constant when moving in the direction
$v$, we get 
\begin{equation}
\nabla g_{l,v,1}(x)=-\frac{\nabla f(z(x))}{\nabla f(z(x))^{T}v}+v.\label{eq:grad-g-l-v-1}
\end{equation}
Similarly, we have $\nabla g_{l,v,2}(x)=\frac{\nabla f(z^{\prime}(x))}{\nabla f(z^{\prime}(x))^{T}v}-v$.
The formula for $\nabla g_{l,v}$ is easily deduced.

Next, we calculate $\nabla^{2}g_{l,v}$ by first calculating $\nabla^{2}g_{l,v,1}$
and $\nabla^{2}g_{l,v,2}$. To reduce notation, we suppress the dependence
of $z$ and $z^{\prime}$ on $x$. Taking the $m$th component of
\eqref{eq:grad-g-l-v-1} gives 
\[
\frac{\partial g_{l,v,1}}{\partial x_{m}}(x)=-\frac{1}{v^{T}\nabla f(z)}\frac{\partial f(z)}{\partial x_{m}}+v_{m},
\]
so

\begin{eqnarray*}
\frac{\partial}{\partial x_{m'}}\left(\frac{\partial g_{l,v,1}}{\partial x_{m}}\right)(x) & = & \frac{-[v^{T}\nabla f(z)]\frac{\partial}{\partial x_{m'}}\left(\frac{\partial f(z)}{\partial x_{m}}\right)+\frac{\partial f(z)}{\partial x_{m}}\left[\frac{\partial}{\partial x_{m'}}(v^{T}\nabla f(z))\right]}{[v^{T}\nabla f(z)]^{2}}.
\end{eqnarray*}

Note that $z(x)=x+g_{l,v,1}(x)v-vv^{T}x$ and $z'(x)=x-g_{l,v,2}(x)v+vv^{T}x$.
We use the notation $1_{\{a=b\}}$ to mean 
\[
1_{\{a=b\}}=\begin{cases}
1 & \mbox{ if }a=b\\
0 & \mbox{ otherwise.}
\end{cases}
\]
So $\frac{\partial z_{k}}{\partial x_{m'}}=1_{\{m'=k\}}+\frac{\partial g_{l,v,1}(x)}{\partial x_{m'}}v_{k}-v_{m^{\prime}}v_{k}$,
and $\frac{\partial z'_{k}}{\partial x_{m'}}=1_{\{m'=k\}}-\frac{\partial g_{l,v,2}(x)}{\partial x_{m'}}v_{k}+v_{m^{\prime}}v_{k}.$
So by the multi-variable chain rule we have 

\begin{eqnarray*}
 &  & \frac{\partial}{\partial x_{m'}}\left(\frac{\partial g_{l,v,1}}{\partial x_{m}}\right)(x)\\
 & = & \frac{-[v^{T}\nabla f(z)]\sum_{k=1}^{n}\left[\frac{\partial^{2}f(z)}{\partial x_{k}\partial x_{m}}\left[1_{\{m'=k\}}+\frac{\partial g_{l,v,1}(x)}{\partial x_{m'}}v_{k}-v_{m^{\prime}}v_{k}\right]\right]}{[v^{T}\nabla f(z)]^{2}}\\
 &  & +\frac{\frac{\partial f(z)}{\partial x_{m}}\sum_{k=1}^{n}\left[v_{k}\sum_{k'=1}^{n}\frac{\partial^{2}f(z)}{\partial x_{k}\partial x_{k'}}\left[1_{\{k'=m'\}}+\frac{\partial g_{l,v,1}(x)}{\partial x_{m'}}v_{k'}-v_{m^{\prime}}v_{k^{\prime}}\right]\right]}{[v^{T}\nabla f(z)]^{2}}.
\end{eqnarray*}

Now we have, 
\begin{eqnarray*}
\nabla^{2}g_{l,v,1}(x) & = & \frac{-\nabla^{2}f(z)-\nabla^{2}f(z)v[\nabla g_{l,v,1}(x)^{T}-v^{T}]}{v^{T}\nabla f(z)}\\
 &  & +\frac{\nabla f(z)v^{T}\nabla^{2}f(z)+\nabla f(z)v^{T}\nabla^{2}f(z)v[\nabla g_{l,v,1}(x)^{T}-v^{T}]}{[v^{T}\nabla f(z)]^{2}}.
\end{eqnarray*}

Substituting $\nabla g_{l,v,1}(x)=-\frac{\nabla f(z)}{v^{T}\nabla f(z)}+v$,
we get:

\begin{eqnarray*}
 &  & \nabla^{2}g_{l,v,1}(x)\\
 & = & -\frac{\nabla^{2}f(z)}{v^{T}\nabla f(z)}+\frac{\nabla^{2}f(z)v\nabla f(z)^{T}+\nabla f(z)v^{T}\nabla^{2}f(z)}{[v^{T}\nabla f(z)]^{2}}\\
 &  & -\frac{\nabla f(z)v^{T}\nabla^{2}f(z)v\nabla f(z)^{T}}{[v^{T}\nabla f(z)]^{3}}.
\end{eqnarray*}
The formula for $\nabla^{2}g_{l,v,2}(x)$ is similar, and the formula
for $\nabla^{2}g_{l,v}(x)$ follows readily.
\end{proof}
The formulas for $\nabla(g_{l,v}^{2})$ and $\nabla^{2}(g_{l,v}^{2})$
can now be calculated, as is done below.
\begin{prop}
\label{pro:Grad-Hess-g-2}(Gradient and Hessian of $g_{l,v}^{2}$)
Given the conditions in Lemma \ref{lem:Grad-Hess-g}, the formulas
for $\nabla(g_{l,v}^{2})(x)$ and $\nabla^{2}(g_{l,v}^{2})(x)$ are
$\nabla(g_{l,v}^{2})(x)=2g_{l,v}(x)\nabla g_{l,v}(x)$ and 
\begin{eqnarray*}
\nabla^{2}(g_{l,v}^{2})(x) & = & 2\left(-\frac{\nabla f(z)}{\nabla f(z)^{T}v}+\frac{\nabla f(z^{\prime})}{\nabla f(z^{\prime})^{T}v}\right)\left(-\frac{\nabla f(z)}{\nabla f(z)^{T}v}+\frac{\nabla f(z^{\prime})}{\nabla f(z^{\prime})^{T}v}\right)^{T}\\
 &  & -2g_{l,v}(x)\left(I-\frac{\nabla f(z)v^{T}}{v^{T}\nabla f(z)}\right)\frac{\nabla^{2}f(z)}{v^{T}\nabla f(z)}\left(I-\frac{\nabla f(z)v^{T}}{v^{T}\nabla f(z)}\right)^{T}\\
 &  & +2g_{l,v}(x)\left(I-\frac{\nabla f(z^{\prime})v^{T}}{v^{T}\nabla f(z^{\prime})}\right)\frac{\nabla^{2}f(z^{\prime})}{v^{T}\nabla f(z^{\prime})}\left(I-\frac{\nabla f(z^{\prime})v^{T}}{v^{T}\nabla f(z^{\prime})}\right)^{T}.
\end{eqnarray*}
\end{prop}
\begin{proof}
We have $\frac{\partial(g_{l,v}^{2})}{\partial x_{m}}(x)=2g_{l,v}(x)\frac{\partial(g_{l,v})}{\partial x_{m}}(x)$,
so $\nabla(g_{l,v}^{2})(x)=2g_{l,v}(x)\nabla g_{l,v}(x)$. Also, 
\[
\frac{\partial}{\partial x_{m'}}\left(\frac{\partial(g_{l,v}^{2})}{\partial x_{m}}(x)\right)=2\frac{\partial g_{l,v}(x)}{\partial x_{m'}}\frac{\partial g_{l,v}(x)}{\partial x_{m}}+2g_{l,v}(x)\frac{\partial^{2}g_{l,v}(x)}{\partial x_{m'}\partial x_{m}}.
\]
We thus have 
\[
\nabla^{2}(g_{l,v}^{2})(x)=2\nabla g_{l,v}(x)\nabla g_{l,v}(x)^{T}+2g_{l,v}(x)\nabla^{2}g_{l,v}(x),
\]
which gives the formula for $\nabla^{2}(g_{l,v}^{2})(x)$. 
\end{proof}
We now discuss the situation when we consider $g_{l,v}$ instead
of its square. Consider a quadratic function $f:\mathbb{R}^{n}\to\mathbb{R}$
whose Hessian has one negative eigenvalue and $n-1$ positive eigenvalues.
For the critical point $\bar{x}$ and critical level $f(\bar{x})$,
a plot of $\lev_{\leq l}f$ for $l<f(\bar{x})$ has two distinct convex
components. One would expect that if $f:\mathbb{R}^{n}\to\mathbb{R}$
is $\mathcal{C}^{2}$ at a nondegenerate saddle point $\bar{x}$ of
Morse index one and $l<f(\bar{x})$, $U\cap\lev_{\leq l}f$ would
consist of two convex components for some neighborhood $U$ of $\bar{x}$.
We have the following result on the convexity of the level sets from
\cite{mountain}.
\begin{prop}
\cite{mountain} \label{pro:2-convex-sets}(Convexity of level sets)
Suppose that $f:\mathbb{R}^{n}\rightarrow\mathbb{R}$ is $\mathcal{C}^{2}$
in a neighborhood of a nondegenerate critical point $\bar{x}$ of
Morse index one. Then if $\epsilon>0$ is small enough, there is a
convex neighborhood $U_{\epsilon}$ of $\bar{x}$ such that $U_{\epsilon}\cap\lev_{\leq f(\bar{x})-\epsilon}f$
is a union of two disjoint convex sets.
\end{prop}
The example below show that Proposition \ref{pro:2-convex-sets} may
be the best possible.
\begin{example}
\label{exa:tightness-extend}(Tightness in Proposition \ref{pro:2-convex-sets})
Figure \ref{fig:nonconvex-level-sets} shows the level set $\lev_{\leq0}f$
for $f:\mathbb{R}^{2}\to\mathbb{R}$ defined by $f(x)=(x_{2}-x_{1}^{2})(x_{1}-x_{2}^{2})$.
For this particular $f$, we have the following.
\begin{enumerate}
\item In Proposition \ref{pro:2-convex-sets}, the neighborhood $U_{\epsilon}$
must satisfy $\diam(U_{\epsilon})\searrow0$ as $\epsilon\searrow0$.
In other words, the dependence of the neighborhood $U_{\epsilon}$
on the parameter $\epsilon$ cannot be lifted. 
\item The level set $\lev_{\leq0}f$ cannot be written as a union of two
convex sets in some neighborhood of $(0,0)$.
\item As a consequence of Proposition \ref{pro:2-convex-sets} and (1),
the function $g_{l,v}:\mathbb{R}^{n}\to\mathbb{R}$ is convex in $x$
in $U_{\epsilon}$ for $l=f(\bar{x})-\epsilon$, but the region on
which $g_{l,v}$ is convex shrinks as $l$ approaches $f(\bar{x})=0$.
\end{enumerate}
\end{example}
\begin{figure}
\includegraphics[scale=0.5]{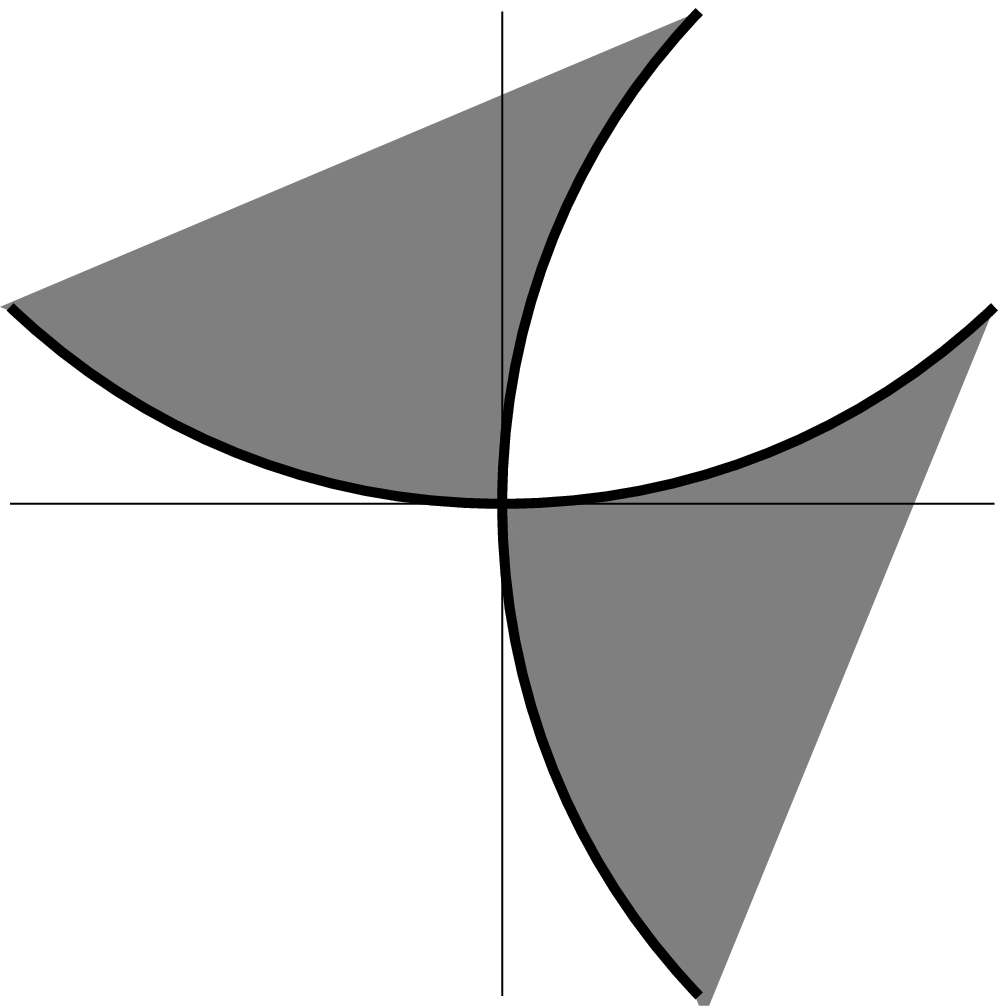}\caption{\label{fig:nonconvex-level-sets}$\lev_{\leq0}f$ for $f(x)=(x_{2}-x_{1}^{2})(x_{1}-x_{2}^{2})$
in Example \ref{exa:tightness-extend}.}
\end{figure}

\section{\label{sec:global-framework}Framework for a mountain pass algorithm}

 In this section, we first present subroutines for a mountain
pass algorithm, and then show how the corresponding mountain pass
algorithm has local and global properties. 

We first present the subroutines that make up the global algorithm.
\begin{algorithm}
\label{alg:subroutines}(Subroutines in global mountain pass algorithm)
Here are the subroutines that will be the building blocks of our global
mountain pass algorithm.
\begin{enumerate}
\item [(PD)] (Parallel distance reduction) Given points $z$ and $z^{\prime}$
and a level $l$ such that $f(z)=f(z^{\prime})=l$,

\begin{enumerate}
\item Let $v=z-z^{\prime}$, and let $x$ be any point on the segment $[z,z^{\prime}]$.
\item From $\nabla f(z)$ and $\nabla f(z^{\prime})$, determine $\nabla(g_{l,v}^{2})(x)$.
The Hessian $\nabla^{2}g_{l,v}^{2}(x)$ may also be calculated or
estimated for a (quasi-) Newton method. These values will give a direction
$d$ for decrease of $g_{l,v}^{2}(\cdot)$. 
\item There is some $t>0$ such that $g_{l,v}(x+td)<g_{l,v}(x)$. Two cases
are possible. If $g_{l,v}(x+td)>0$, then $z(x+td)$ and $z^{\prime}(x+td)$
are new iterates reducing the parallel distance. If $g_{l,v}(x+td)=0$,
then let $x^{\prime}$ be a local maximum of $f$ on the line $\{x+td\}+\mathbb{R}\{v\}$.
We have $f(x^{\prime})\leq l$, and we should run $(l\downarrow)$
below.
\end{enumerate}
\item [(Av)] (Adjusting vector $v$) Given points $z$ and $z^{\prime}$
and a level $l$ such that $f(z)=f(z^{\prime})=l$,

\begin{enumerate}
\item Perturb $z$ and/or $z^{\prime}$ such that we still have $f(z)=f(z^{\prime})=l$,
and that $\|z-z^{\prime}\|$ is reduced. The vector $v=z-z^{\prime}$
is now adjusted.
\end{enumerate}
\item [($l\downarrow$)] (Decrease level $l$) Given $x$ and $v\neq0$
such that $x$ is a local maximum of $f$ on $\{x\}+\mathbb{R}\{v\}$, 

\begin{enumerate}
\item Find local minimizer of $f$ on $\{x\}+\mathbb{R}\{d\}$, where $d\perp v$
and $-\nabla f(x)^{T}d<0$. The direction $d$ can be chosen to be
the projection of $-\nabla f(x)$ onto the subspace perpendicular
to $v$. 
\end{enumerate}
\item [($l\uparrow$)] (Increase level $l$) Given points $z$ and $z^{\prime}$
and a level $l$ such that $f(z)=f(z^{\prime})=l$,

\begin{enumerate}
\item Choose some $x\in[z,z^{\prime}]$ such that $f(x)>l$. (One choice
is $x=\frac{1}{2}(z+z^{\prime})$.) Perturb $z$ and $z^{\prime}$
so that $f(z)$ and $f(z^{\prime})$ equal this new value of $l$. 
\end{enumerate}
\end{enumerate}
\end{algorithm}
Other ways of adjusting the vector $v$ apart from (Av) are possible,
though they are not as simple as (Av). For example, the vector $v$
can also be calculated by taking the eigenvector corresponding to
the negative eigenvalue of $\nabla^{2}f(z_{i})$, $\nabla^{2}f(z_{i}^{\prime})$,
or some combination of the two matrices. 

We gave a method of decreasing the level $l$ in ($l\downarrow$).
Adjustments to the strategy presented in ($l\downarrow$) can be made
as needed. For example, the condition $d\perp v$ can be adjusted. 

There are also other reasons to adjust $l$. First, the contrapositive
of Lemma \ref{lem:unif-terms}(1) later can be roughly interpreted
as follows: If $1/|v^{T}u(\nabla f(z))|$ is too small, then the critical
level is below $f(z)=l$. We can thus reduce the level $l$. Secondly,
when $g_{l,v}(x)$ is too high, signifying that the points $z(x)$
and $z^{\prime}(x)$ are too far apart, one can increase $l$. Third,
the points evaluated may not have function value $l$, making a different
value more suitable. Lastly, it is possible to estimate $l$ by setting
the minimizer of $g_{l,v}(x)^{2}$ to be zero from the formula in
\eqref{eq:g-l-v-full}.

\subsection{Fast local convergence }

We discuss the fast local convergence properties of the level set
algorithm. We recall our mountain pass algorithm in \cite{mountain},
where we proved local superlinear convergence of a level set algorithm
under restrictive assumptions, and show how the difficult steps there
can be seen as limiting cases of subroutines (Av) and ($l\downarrow$).

We recall our mountain pass algorithm in \cite{mountain}.
\begin{algorithm}
\label{alg:old-alg}\cite{mountain} (A local superlinearly convergent
algorithm) Let counter $i$ be $0$. Given points $z_{0}$ and $z_{0}^{\prime}$,
and a level $l_{0}$ such that $f(z_{0})=f(z_{0}^{\prime})=l_{0}$.
Let $U$ be an open neighborhood of the saddle point $\bar{x}$ that
contains $z_{0}$ and $z_{0}^{\prime}$. 
\begin{enumerate}
\item Perturb $z_{i}$ and $z_{i}^{\prime}$ to the points $\tilde{z}_{i}$
and $\tilde{z}_{i}^{\prime}$ so that for some open set $U$, $\tilde{z}_{i}$
and $\tilde{z}_{i}^{\prime}$ are the minimizers of the problem 
\begin{align}
\min_{x,y} & \|x-y\|\nonumber \\
\mbox{s.t. } & x,y\mbox{ lie in the same component }U\cap\lev_{\leq l_{i}}f\mbox{ as }z_{i}\mbox{ and }z_{i}^{\prime}\mbox{ respectively.}\label{eq:alt-proj}
\end{align}

\item Let $v_{i}$ be the unit vector in the same direction as $\tilde{z}_{i}-\tilde{z}_{i}^{\prime}$.
Find the minimum of $f$ on $U\cap L_{i}$, where $L_{i}$ is the
perpendicular bisector of $\tilde{z}_{i}$ and $\tilde{z}_{i}^{\prime}$.
Let this value be $l_{i+1}$. Find $z_{i+1}$ and $z_{i+1}^{\prime}$
such that they are points in the same components of the level set
$U\cap\lev_{\leq l_{i+1}}f$ as $\tilde{z}_{i}$ and $\tilde{z}_{i}^{\prime}$
respectively, and that $z_{i+1}-z_{i+1}^{\prime}$ points in the same
direction as $v_{i}$. 
\item Stop if $\|z_{i+1}-z_{i+1}^{\prime}\|$ is sufficiently small, or
until we find a point $x$ such that $\|\nabla f(x)\|$ is sufficiently
small. Increase the counter $i$, and return to step 1.
\end{enumerate}
\end{algorithm}
Algorithm \ref{alg:old-alg} can be built from the subroutines highlighted
in Algorithm \ref{alg:subroutines}. Step (1) can be seen as applying
the step (Av) infinitely many times, while step (2) can be seen as
applying one step of ($l\uparrow$), then applying ($l\downarrow$)
infinitely often till the minimizer of $f$ on $U\cap L_{i}$ is reached.

The main result in \cite{mountain} is that in some neighborhood $U$
of a nondegenerate saddle point $\bar{x}$ of Morse index one, the
steps in Algorithm \ref{alg:old-alg} are well defined, and Algorithm
\ref{alg:old-alg} converges locally superlinearly to $\bar{x}$.
This shows that level set methods can satisfy Principle (P1).

However, Algorithm \ref{alg:old-alg} has some disadvantages:
\begin{itemize}
\item [(D1)] Step 1 in Algorithm \ref{alg:old-alg} is difficult to perform
in practice. If an alternating projection method was used to solve
\eqref{eq:alt-proj} for example, the convergence will be very slow
when close to the minimizers.
\begin{figure}
\includegraphics[scale=0.4]{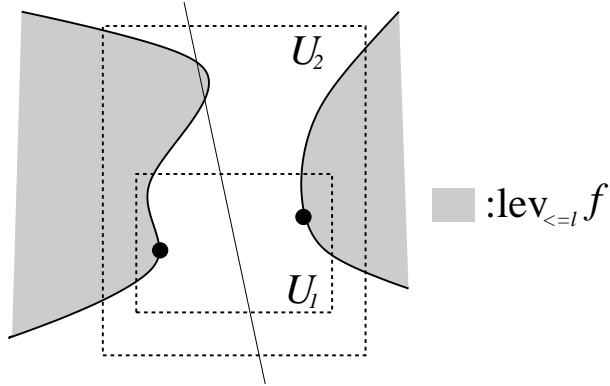}\caption{\label{fig:D2}We elaborate on the possible difficulties in finding
a lower bound of critical level explained in step 2 of Algorithm \ref{alg:old-alg}.
Let $L$ be the perpendicular bisector of the two closest points as
shown. The neighborhood $U_{1}$ is too small as a minimizer of $f$
on $U_{1}\cap L$ does not exist in the relative interior of $U_{1}\cap L$.
The neighborhood $U_{2}$ is too large since the minimum value of
$f$ on $U_{2}\cap L$ is worse than the previous lower bound on the
critical value. }
\end{figure}

\item [(D2)] Related to (D1) is the problem of ensuring that $U\cap\lev_{\leq l}f$
is a union of two components for some convex neighborhood $U$ of
$\bar{x}$. This in turn requires $l$ to satisfy $l<f(\bar{x})$,
where $f(\bar{x})$ is the critical level. Step 2 in Algorithm \ref{alg:old-alg}
ensures that the calculated level is an underestimate of the critical
level, but this step may involve more computational effort than is
necessary.
\end{itemize}
Algorithm \ref{alg:old-alg} can be extended to a global algorithm.
A few problems may arise in the global case. Firstly, the problem
of minimizing $f$ on $L_{i}$ is not necessarily easy. Sometimes,
$f$ may not have a local minimizer in $U\cap L_{i}$. Secondly, the
new estimate $l$ of the critical level $f(\bar{x})$ may be even
lower than the previous estimate, rendering it useless as a lower
bound on $f(\bar{x})$. Lastly, the estimate $l$ of the critical
level may actually be an upper estimate of $f(\bar{x})$ instead.
See Figure \ref{fig:D2}.

Proposition \ref{pro:quad-formula} suggests that using $g_{l,v}^{2}(\cdot)$
overcomes the difficulties (D1) and (D2).  Provided $v$ is close
enough to the eigenvector corresponding to the negative eigenvalue
of $\nabla f(\bar{x})$, the function $g_{l,v}^{2}(\cdot)$ restricted
to any $(n-1)$ dimensional affine space not containing $v$ is the
maximum of a quadratic with positive definite Hessian and $0$. One
can first minimize $g_{l,v}^{2}(\cdot)$ as a quadratic. Once close
enough to $\bar{x}$, the minimizer of the corresponding quadratic,
say $\tilde{x}$, will  give a good estimate of $\bar{x}$.

\subsection{Global convergence results}

We now look at the global mountain pass algorithm involving the subalgorithms
listed in Algorithm \ref{alg:subroutines}. 
\begin{algorithm}
\label{alg:new-alg}(Global mountain pass algorithm) Let the counter
$i$ be $0$. Suppose the points $z_{0}$ and $z_{0}^{\prime}$ and
a level $l_{0}$ are such that $f(z_{0})=f(z_{0}^{\prime})=l_{0}$.
Let $x_{0}$ be some point in the line segment $[z_{0},z_{0}^{\prime}]$.
Let $v_{0}=z_{0}-z_{0}^{\prime}$.
\begin{enumerate}
\item Run (PD) on $z_{i}$, $z_{i}^{\prime}$, $v_{i}$ and $l_{i}$. Three
outcomes are possible:

\begin{enumerate}
\item If the new parallel distance is positive and sufficient decrease in
the parallel distance is obtained, let the output be $z_{i+1}$ and
$z_{i+1}^{\prime}$. Let $l_{i+1}=l_{i}$. Run (Av), which perturbs
either $z_{i+1}$ or $z_{i+1}^{\prime}$. The vector $v_{i+1}$ is
set to be the unit vector in the direction of $z_{i+1}-z_{i+1}^{\prime}$.
\item If the new parallel distance is positive but the parallel distance
changed little from previous iterations, run ($l\uparrow$) to perturb
$z_{i+1}$ and $z_{i+1}^{\prime}$, and let $l_{i+1}$ be the new
level. The vector $v_{i+1}$ equals $v_{i}$, unchanged from before.
\item If the new parallel distance is zero, then let $l_{i+1}$ be the new
level, and let $x_{i+1}$ be the local maximum as stated in (PD).
Run ($l\downarrow$). The new level is still labeled as $l_{i+1}$.
The vector $v_{i+1}$ equals $v_{i}$, unchanged from before.
\end{enumerate}
\item Increase $i$ by one. If in the course of the calculations, a point
$x$ such that $\|\nabla f(x)\|$ is small is encountered, then the
algorithm ends. If $\|z_{i}-z_{i}^{\prime}\|$ is small and the distance
of $0$ to the convex hull of $\{\nabla f(z_{i}),\nabla f(z_{i}^{\prime})\}$
is small, then we can extrapolate some point $x\in[z_{i},z_{i}^{\prime}]$
for which $\|\nabla f(x)\|$ is small, and end the algorithm. Otherwise,
go back to step 1.
\end{enumerate}
\end{algorithm}

In Algorithm \ref{alg:new-alg}, the subroutines (PD) and (Av) reduce
the distance between the components of the level set $U\cap\lev_{\leq l}f$.
 Algorithm \ref{alg:new-alg} illustrates just one way to decide
which of the subroutines (PD), (Av), ($l\uparrow$) and ($l\downarrow$)
to use at each step, and other combinations are possible. There is
still flexibility on whether option 1(a) or 1(b) is taken. Once close
enough to the saddle point, a quadratic model method can be used. 

The basis of (P2) for both Algorithms \ref{alg:old-alg} and \ref{alg:new-alg}
is the following result.
\begin{thm}
\label{thm:conv-saddle}\cite{mountain}(Global convergence of level
set algorithm) Let $f:X\to\mathbb{R}$. Suppose $\{a_{i}\}_{i=0}^{\infty}$
and $\{b_{i}\}_{i=0}^{\infty}$ are sequences of points and $\{l_{i}\}_{i=0}^{\infty}$
is a sequence satisfying $l_{i}\nearrow f(\bar{x})$. If $x_{i}$
and $y_{i}$ lie in separate components of $\{x\mid f(x)\leq l_{i}\}$,
and $\bar{x}=\lim_{i\to\infty}a_{i}=\lim_{i\to\infty}b_{i}$, then
$\bar{x}$ is a saddle point.
\end{thm}
One difficulty is to decide whether $a_{i}$ and $b_{i}$ are in different
components of $\lev_{\leq l_{i}}f$, but we can use $\nabla f(a_{i})$
and $\nabla f(b_{i})$ to make a guess. Note that provided the limits
exist, $\lim_{i\to\infty}a_{i}=\lim_{i\to\infty}b_{i}$ is equivalent
to $\lim_{i\to\infty}\|a_{i}-b_{i}\|=0$. This principle can be seen
as a convergence property of Algorithms \ref{alg:old-alg} and \ref{alg:new-alg}.
It is therefore pragmatic to decrease the distance or parallel distance
between the components of the level sets, especially at the start
of a global mountain pass algorithm where the quadratic approximation
is not valid yet. The problem of choosing the sequence $\{l_{i}\}_{i=0}^{\infty}$
is much more difficult. The strategy in Algorithm \ref{alg:new-alg}
is adequate for our numerical experiment, but more still needs to
be done.

\section{\label{sec:Convexity-para-dist}Independence of $l$ in estimating
$\nabla^{2}(g_{l,v}^{2})(\cdot)$}

We recall that in our level set algorithm in Section \ref{sec:global-framework},
we perturb the level $l$ using subroutines $(l\uparrow)$ and $(l\downarrow)$
so that $l$ converges to the critical value $f(\bar{x})$ of the
saddle point $\bar{x}$. Such changes in $l$ can be quite sudden.
The Hessian $\nabla^{2}(g_{l,v}^{2})(\cdot)$ is not continuous at
$\bar{x}$ because of the $\frac{1}{v^{T}f(z)}$ in its formula, and
the continuity at an $x$ where $\nabla^{2}(g_{l,v}^{2})(\cdot)$
is only good enough for small changes in $l$. In this section, we
show in Theorem \ref{thm:convexity-smooth} that there is a neighborhood
$U$ of the saddle point $\bar{x}$ such that as long as $l<f(\bar{x})$
is close enough to $f(\bar{x})$ and $v$ is close enough to the eigenspace
corresponding to the negative eigenvalue of $\nabla^{2}(g_{l,v}^{2})(\bar{x})$,
the Hessian $\nabla^{2}(g_{l,v}^{2})(x)$ for $x\in U$ can be estimated
from a quadratic model of $f$ at $\bar{x}$. Such a result shows
that under changes of $l$ near $\bar{x}$, the Hessian $\nabla^{2}(g_{l,v}^{2})(\cdot)$
does not depend too much on $l$, making previous estimates of $\nabla^{2}(g_{l,v}^{2})(\cdot)$
useful for future iterations. As a consequence, we obtain the convexity
of $g_{l,v}(\cdot)^{2}$. 

First, we have the following result that allows us to identify convexity.
\begin{prop}
\label{pro:convex-from-pd}(Convexity from positive definite Hessians)
Suppose $f:\mathbb{R}^{n}\to[0,\infty)$ is a continuous function
that is $\mathcal{C}^{2}$ at all points $x$ satisfying $f(x)>0$,
and the corresponding Hessian $\nabla^{2}f(x)$ is positively semidefinite.
Then $f$ is convex. (The issue here is that the nonsmoothness of
$f$ on the boundary of $\{x\mid f(x)=0\}$ does not affect convexity.)\end{prop}
\begin{proof}
The usual convexity test $tf(x)+(1-t)f(y)\geq f(tx+(1-t)y)$ for all
$x,y\in\mathbb{R}^{n}$ and $t\in(0,1)$ allows us to reduce the problem
in $\mathbb{R}^{n}$ to that of $n=1$. We first notice that there
cannot exist $x_{1},x_{2}\in\mathbb{R}$ such that $x_{1}<\tilde{x}<x_{2}$,
$f(x_{1})=f(x_{2})=0$, and $f(x)>0$ for all $x\in(x_{1},x_{2})$,
since this is a contradiction to the convexity of $f$ on $(x_{1},x_{2})$. 

Using the above property, we can find $x_{3},x_{4}\in\mathbb{R}\cup\{-\infty,\infty\}$
such that $x_{3}\leq x_{4}$ and 
\[
f(x)\ \begin{cases}
=0 & \mbox{ if }x\in[x_{3},x_{4}]\\
>0 & \mbox{ if }x\notin[x_{3},x_{4}].
\end{cases}
\]
Note that one or both of $x_{3}$ and $x_{4}$ might be $\pm\infty$.
It is an easy exercise that the subdifferential mapping $\partial f$
is monotone, thus $f$ is convex.
\end{proof}
We shall make use of Proposition \ref{pro:convex-from-pd} to establish
the convexity of $g_{l,v}^{2}$ by making sure that the Hessian $\nabla^{2}(g_{l,v}^{2})$
is positive semidefinite whenever $g_{l,v}>0$. 

We make some simplifying assumptions for the rest of this section.
\begin{assumption}
\label{ass:smooth-f}(Smooth $f$) Assume that $f:\mathbb{R}^{n}\to\mathbb{R}$
is a \emph{$\mathcal{C}^{2}$ function with a nondegenerate critical
point $\bar{x}=0$ of Morse index one satisfying $f(0)=0$ such that
$H=\nabla^{2}f(0)$ is diagonal with entries arranged in decreasing
manner as $\lambda_{1},\dots,\lambda_{n}$. This means that the diagonal
entries of $\nabla^{2}f(0)$ consist of $n-1$ positive eigenvalues
and one negative eigenvalue. Let the eigenvector corresponding to
the negative eigenvalue $\lambda_{n}$ be $\bar{v}$. }
\end{assumption}
We also make another definition that will simplify many of the statements
in this section. Denote $\bar{f}_{\delta}:\mathbb{R}^{n}\to\mathbb{R}$
by 
\begin{equation}
\bar{f}_{\delta}(x)=\frac{1}{2}x^{T}[\nabla^{2}f(0)+\delta I]x.\label{eq:bar-f-2}
\end{equation}
Let $\bar{g}_{\delta,l,v}$ be the value of $g_{l,v}$ defined through
the quadratic $\bar{f}_{\delta}(\cdot)$ (instead of through $f(\cdot)$).
 The values $\bar{z}_{\delta}(x)$ and $\bar{z}_{\delta}^{\prime}(x)$,
defined through $\bar{f}_{\delta}$ will be of use later in this section.
 We write $\bar{f}=\bar{f}_{0}$, $\bar{g}_{l,v}=\bar{g}_{0,l,v}$,
$\bar{z}(\cdot)=\bar{z}_{0}(\cdot)$ and $\bar{z}^{\prime}(\cdot)=\bar{z}_{0}^{\prime}(\cdot)$.
We also write $H_{\delta}=\nabla^{2}f(0)+\delta I$ to simplify notation.
\begin{defn}
(Continuity condition) For a function $f:\mathbb{R}^{n}\to\mathbb{R}$
satisfying Assumption \ref{ass:smooth-f}, $\delta>0$, $\gamma>0$
and convex neighborhoods $U_{\delta}$ and $U_{\delta}^{\prime}$
of $0$ such that $U_{\delta}\subset U_{\delta}^{\prime}$, we say
that condition $P(f,\delta,\gamma,U_{\delta},U_{\delta}^{\prime})$
is satisfied if 
\begin{enumerate}
\item $\|\nabla^{2}f(x)-\nabla^{2}f(0)\|\leq\delta$ for all $x\in U_{\delta}^{\prime}$,
\item For $\bar{f}_{\delta}:\mathbb{R}^{n}\to\mathbb{R}$ and $\bar{S}_{\delta,l,v}(\cdot)$
as defined in \eqref{eq:bar-f}, we have $\bar{S}_{\delta,l,v}(x)\subset U_{\delta}^{\prime}$
for all $x\in U_{\delta}$ and $l\in(-\gamma,0]$.
\end{enumerate}
\end{defn}
It is clear through Proposition \ref{pro:Well-defined-g} and the
continuity of the Hessian that for any $\delta>0$, there must be
convex neighborhoods $U_{\delta}$ and $U_{\delta}^{\prime}$ such
that $U_{\delta}\subset U_{\delta}^{\prime}$ and $P(f,\delta,\gamma,U_{\delta},U_{\delta}^{\prime})$
holds. It is also clear that if $P(f,\delta,\gamma,U_{\delta},U_{\delta}^{\prime})$
holds, we have 
\begin{equation}
\frac{1}{2}x^{T}[\nabla^{2}f(0)-\delta I]x\leq f(x)\leq\frac{1}{2}x^{T}[\nabla^{2}f(0)+\delta I]x\mbox{ for all }x\in U_{\delta}.\label{eq:u-and-l-bdd-of-f}
\end{equation}
The next result is a bound on  the error in $z(x)$.
\begin{lem}
\label{lem:control-z(x)}(Controlling $\bar{z}(x)$) Suppose that
$f:\mathbb{R}^{n}\to\mathbb{R}$ is $\mathcal{C}^{2}$ and satisfies
Assumption \ref{ass:smooth-f}. Let $v$ be a unit vector such that
$v^{T}\nabla^{2}f(0)v<0$. For any $\epsilon>0$, there are $\delta>0$,
$\gamma>0$ and convex neighborhoods $U_{\delta}$ and $U_{\delta}^{\prime}$
of $0$ such that $P(f,\delta,\gamma,U_{\delta},U_{\delta}^{\prime})$
holds, and for all $x\in U_{\delta}$ and $l\in(-\gamma,0]$, we have
\begin{eqnarray*}
\|\bar{z}(x)-z(x)\| & \leq & \epsilon\|\bar{z}(x)\|,\\
\|\bar{z}^{\prime}(x)-z^{\prime}(x)\| & \leq & \epsilon\|\bar{z}(x)\|,\\
\mbox{and }|\bar{g}_{l,v}(x)-g_{l,v}(x)| & \leq & \epsilon\|\bar{z}(x)\|.
\end{eqnarray*}
\end{lem}
\begin{proof}
Since $z(x)$ and $\bar{z}(x)$ lie inside the line segment $[\bar{z}_{\delta}(x),\bar{z}_{-\delta}(x)]$,
we have $\|\bar{z}(x)-z(x)\|\leq\|\bar{z}_{\delta}(x)-\bar{z}_{-\delta}(x)\|$.
Since $\bar{z}_{\delta}(x)$, $\bar{z}_{-\delta}(x)$, $\bar{z}_{-\delta}^{\prime}(x)$
and $\bar{z}_{\delta}^{\prime}(x)$ line up in a line (with direction
$v$) in that order, we have 
\begin{eqnarray*}
\|\bar{z}_{\delta}(x)-\bar{z}_{-\delta}(x)\| & < & \|\bar{z}_{\delta}(x)-\bar{z}_{-\delta}(x)\|+\|\bar{z}_{\delta}^{\prime}(x)-\bar{z}_{-\delta}^{\prime}(x)\|\\
 & = & \bar{g}_{\delta,l,v}(x)-\bar{g}_{-\delta,l,v}(x).
\end{eqnarray*}
Similarly,
\begin{eqnarray*}
|\bar{g}_{l,v}(x)-g_{l,v}(x)| & < & \|\bar{z}_{\delta}(x)-\bar{z}_{-\delta}(x)\|+\|\bar{z}_{\delta}^{\prime}(x)-\bar{z}_{-\delta}^{\prime}(x)\|\\
 & = & \bar{g}_{\delta,l,v}(x)-\bar{g}_{-\delta,l,v}(x).
\end{eqnarray*}
Our goal is therefore to prove that for every $\epsilon>0$, we can
find a $\delta>0$ such that $\bar{g}_{\delta,l,v}(x)-\bar{g}_{-\delta,l,v}(x)\leq\epsilon\|\bar{z}(x)\|$
for all $x\in\mathbb{R}^{n}$. 

Note that our problem has now been transformed to a new problem on
an exact quadratic $\bar{f}(\cdot)$. 

The treatment for the case $l=0$ and $l<0$ are different, and we
start off by treating the case $l=0$. 

\textbf{CASE $l=0$: }For a point $x\in\mathbb{R}^{n}$, the sets
$\lev_{\leq0}\bar{f}_{\delta}$ and $\lev_{\leq0}\bar{f}_{-\delta}$
are cones, with $\lev_{\leq0}\bar{f}_{\delta}\subset\lev_{\leq0}\bar{f}_{-\delta}$.
For $\delta>0$ small enough, $\nabla^{2}f(0)$ consists of $n-1$
positive eigenvalues and one negative eigenvalue, so $\lev_{\leq0}\bar{f}_{\delta}$
and $\lev_{\leq0}\bar{f}_{-\delta}$ are both the union of two convex
cones intersecting only at $0$. For a point $x$, the points $\bar{z}_{\delta}(x)$,
$\bar{z}_{-\delta}(x)$, $\bar{z}_{\delta}^{\prime}(x)$ and $\bar{z}_{-\delta}^{\prime}(x)$
can be calculated easily from the quadratic formulas we have seen
in the proof of previous results (in particular, Proposition \ref{pro:quad-formula}),
giving 
\begin{eqnarray*}
\bar{g}_{\delta,l,v}(x) & = & \|\bar{z}_{\delta}(x)-\bar{z}_{\delta}^{\prime}(x)\|=\frac{2\sqrt{[v^{T}H_{\delta}x]^{2}-[v^{T}H_{\delta}v][x^{T}H_{\delta}x]}}{v^{T}H_{\delta}v}\\
\mbox{and }\bar{g}_{-\delta,l,v}(x) & = & \|\bar{z}_{-\delta}(x)-\bar{z}_{-\delta}^{\prime}(x)\|=\frac{2\sqrt{[v^{T}H_{-\delta}x]^{2}-[v^{T}H_{-\delta}v][x^{T}H_{-\delta}x]}}{v^{T}H_{-\delta}v}.
\end{eqnarray*}
Consider the problem 
\[
\max_{x\in\partial\mathbb{B}}h_{\delta}(x),
\]
where $h_{\delta}(x)=\bar{g}_{\delta,l,v}(x)-\bar{g}_{-\delta,l,v}(x)$.
The function $h_{\delta}(\cdot)$ is continuous, and the set $\partial\mathbb{B}:=\{x:\|x\|=1\}$
is compact. The optimization problem above satisfies the conditions
in Proposition \ref{pro:conv-to-zero}, so for any $\epsilon>0$,
we can choose $\delta>0$ such that $\max_{x\in\partial\mathbb{B}}h_{\delta}(x)<\epsilon$.
We have 
\[
h_{\delta}(\bar{z}(x))\leq\|\bar{z}(x)\|\max_{y\in\partial\mathbb{B}}h_{\delta}(y)<\epsilon\|\bar{z}(x)\|.
\]

\textbf{CASE $l<0$:} We can consider the case $l=-1/2$ first. The
other cases follow by a scaling. 

If $\|x\|>1/\sqrt{\delta}$, then $\bar{f}_{2\delta}(x)\leq0$ implies
that $\bar{f}_{\delta}(x)=\bar{f}_{2\delta}(x)-\frac{1}{2}\delta\|x\|^{2}\leq-\frac{1}{2}$.
Also, $\bar{f}_{-\delta}(x)\leq-\frac{1}{2}$ clearly implies $\bar{f}_{-2\delta}(x)\leq\bar{f}_{-\delta}(x)\leq0$.
This gives 
\begin{eqnarray*}
\left[\frac{1}{\sqrt{\delta}}\mathbb{B}\right]^{C}\cap\lev_{\leq0}\bar{f}_{2\delta} & \subset & \left[\frac{1}{\sqrt{\delta}}\mathbb{B}\right]^{C}\cap\lev_{\leq-1/2}\bar{f}_{\delta}\\
 & \subset & \left[\frac{1}{\sqrt{\delta}}\mathbb{B}\right]^{C}\cap\lev_{\leq-1/2}\bar{f}_{-\delta}\\
 & \subset & \left[\frac{1}{\sqrt{\delta}}\mathbb{B}\right]^{C}\cap\lev_{\leq0}\bar{f}_{-2\delta},
\end{eqnarray*}
 where $[\cdot]^{C}$ is the complementation of a set. By the treatment
for the case $l=0$, for any $\epsilon>0$, we can find $\delta_{1}>0$
such that if $\|\bar{z}(x)\|>1/\sqrt{\delta_{1}}$, then $\|\bar{z}_{\delta_{1}}(x)-\bar{z}_{-\delta_{1}}(x)\|<\frac{\epsilon}{2}\|\bar{z}(x)\|$.
Therefore, if $\delta_{2}\in(0,\delta_{1})$, we have $\|\bar{z}_{\delta_{2}}(x)-\bar{z}_{-\delta_{2}}(x)\|<\frac{\epsilon}{2}\|\bar{z}(x)\|$,
which gives 
\[
\bar{g}_{\delta,l,v}(x)-\bar{g}_{-\delta,l,v}(x)\leq\epsilon\|\bar{z}(x)\|.
\]
We still need to treat the case where $\|\bar{z}(x)\|\leq1/\sqrt{\delta_{1}}$.
The condition $\bar{f}(\bar{z}(x))=-\frac{1}{2}$ implies that $\|\bar{z}(x)\|\geq1/\sqrt{-\lambda_{n}+2\delta_{1}}$,
where $\lambda_{n}$ is the negative eigenvalue of $\nabla^{2}f(0)$.
We make use of the same strategy to estimate $\|\bar{z}_{\delta}(x)-\bar{z}_{-\delta}(x)\|$
as in the last case. This time, the formulas give 
\begin{eqnarray*}
\bar{g}_{\delta,l,v}(x)-\bar{g}_{-\delta,l,v}(x) & = & \frac{2\sqrt{[v^{T}H_{\delta}x]^{2}-[v^{T}H_{\delta}v][x^{T}H_{\delta}x-2l]}}{v^{T}H_{\delta}v}\\
 &  & -\frac{2\sqrt{[v^{T}H_{-\delta}x]^{2}-[v^{T}H_{-\delta}v][x^{T}H_{-\delta}x-2l]}}{v^{T}H_{-\delta}v}.
\end{eqnarray*}
We are led to consider the problem
\[
\max_{x\in C}h_{\delta}(x),
\]
where 
\begin{eqnarray*}
h_{\delta}(x) & := & \bar{g}_{\delta,l,v}(x)-\bar{g}_{-\delta,l,v}(x),\\
\mbox{ and }C & = & \big\{ y:1/\sqrt{-\lambda_{n}+2\delta_{1}}\leq\|y\|\leq1/\sqrt{\delta_{1}}\big\}.
\end{eqnarray*}
Once again, $h_{\delta}(\cdot)$ is continuous, $C$ is compact, and
Proposition \ref{pro:conv-to-zero} can be applied. There is some
$\delta_{2}$ such that $0<\delta_{2}<\delta_{1}$ and $\max_{x\in C}h_{\delta_{2}}(x)<\epsilon/\sqrt{-\lambda_{n}+2\delta_{1}}$.
If $\|\bar{z}(x)\|\in[1/\sqrt{-\lambda_{n}+2\delta_{1}},1/\sqrt{\delta_{1}}]$,
then we have 
\[
h_{\delta_{2}}\big(\bar{z}(x)\big)\leq\epsilon/\sqrt{-\lambda_{n}+2\delta_{1}}\leq\epsilon\|\bar{z}(x)\|.
\]
The case where $l$ is another negative number differ from the case
$l=-1/2$ by a scaling. Our claim follows.
\end{proof}
Here is a result that we have used for Lemma \ref{lem:control-z(x)}.
\begin{prop}
\label{pro:conv-to-zero}(Convergence to zero of maximum value) Suppose
that $h_{\delta}:C\to\mathbb{R}$ is continuous for all $\delta\geq0$,
$C$ is compact, and that $\delta_{1}<\delta_{2}$ implies $h_{\delta_{1}}(x)\leq h_{\delta_{2}}(x)$
for all $x\in C$. Assume also that for all $x\in C$, $h_{\delta}(x)\searrow0$
as $\delta\searrow0$. Then $\max_{x\in C}h_{\delta}(x)\searrow0$
as $\delta\searrow0$. \end{prop}
\begin{proof}
For each sequence $\delta_{i}\searrow0$ as $i\nearrow\infty$, there
is a maximizer $\tilde{x}_{i}$ such that $h_{\delta_{i}}(\tilde{x}_{i})=\max_{x\in C}h_{\delta_{i}}(x)$.
It suffices to show that $h_{\delta_{i}}(\tilde{x}_{i})\searrow0$
as $i\nearrow\infty$. Due to the compactness of $C$, we can assume
that there is a subsequence of $\{\tilde{x}_{i}\}$ converging to
some $\tilde{x}\in C$. For any $\epsilon>0$, there is some $a$
such that $h_{\delta_{a}}(\tilde{x})<\epsilon$ and a neighborhood
$U_{\epsilon}$ of $\tilde{x}$ such that $h_{\delta_{a}}(x)<2\epsilon$
for all $x\in U_{\epsilon}$. This means that some tail of the sequence
$\{h_{\delta_{i}}(\tilde{x}_{i})\}_{i=1}^{\infty}$ is less than $2\epsilon$.
Since $\epsilon$ is arbitrary, $\max_{x\in C}h_{\delta_{i}}(x)=h_{\delta_{i}}(\tilde{x}_{i})\searrow0$
as $i\nearrow\infty$ as needed.
\end{proof}
For $x\neq0$, let $u(x)=x/\|x\|$. Here are some bounds we need to
check:
\begin{lem}
\label{lem:unif-terms}(Uniform bounds on terms) Suppose that $\bar{f}:\mathbb{R}^{n}\to\mathbb{R}$
is defined by $\bar{f}(x)=\frac{1}{2}x^{T}\nabla^{2}f(0)x$, where
$f:\mathbb{R}^{n}\to\mathbb{R}$ is a function satisfying Assumption
\ref{ass:smooth-f}. Let $\bar{v}$ be the eigenvector corresponding
to the negative eigenvalue of $\nabla^{2}f(0)$. Assume that $v$
is a unit vector such that $\|v-\bar{v}\|<\alpha$. Let $\bar{z}$
be a point such that $f(\bar{z})\leq0$. We have the following:
\begin{enumerate}
\item $1/|v^{T}u(\nabla\bar{f}(\bar{z}))|<1/\left[\sqrt{\frac{-\lambda_{n}}{\max(\lambda_{1},-\lambda_{n})-\lambda_{n}}}-\alpha\right]$
for all $\bar{z}$ such that $\bar{f}(\bar{z})\leq0$.
\item $\frac{\bar{g}_{l,v}(\bar{z})}{\|\nabla\bar{f}(\bar{z})\|}\leq\frac{2}{|\lambda_{n}|-2\alpha|\lambda_{n}|-\alpha^{2}\max(\lambda_{1},-\lambda_{n})}$
for all $\bar{z}$ satisfying $f(\bar{z})=l$, where $l\leq0$. 
\end{enumerate}
\end{lem}
\begin{proof}
Let $H$ be $\nabla^{2}f(0)$, which we recall is diagonal. We prove
(1) and (2). 

(1) We have 
\begin{eqnarray*}
|v^{T}u(\nabla\bar{f}(\bar{z}(x)))| & \geq & |\bar{v}^{T}u(\nabla\bar{f}(\bar{z}))|-|(v-\bar{v})^{T}u(\nabla\bar{f}(\bar{z}))|\\
 & \geq & \left|\bar{v}^{T}\frac{H\bar{z}}{\|H\bar{z}\|}\right|-\alpha\\
 & = & \left|\frac{\lambda_{n}\bar{z}_{n}}{\sqrt{\sum_{i=1}^{n}\lambda_{i}^{2}\bar{z}_{i}^{2}}}\right|-\alpha.
\end{eqnarray*}
 From the fact that $f(\bar{z})\leq0$, we have 
\begin{eqnarray*}
\sum_{i=1}^{n}\lambda_{i}\bar{z}_{i}^{2} & \leq & 0\\
\Rightarrow\lambda_{n}^{2}\bar{z}_{n}^{2} & \geq & -\lambda_{n}\sum_{i=1}^{n-1}\lambda_{i}\bar{z}_{i}^{2}.
\end{eqnarray*}
Therefore
\begin{eqnarray*}
\left(\frac{\lambda_{n}\bar{z}_{n}}{\sqrt{\sum_{i=1}^{n}\lambda_{i}^{2}\bar{z}_{i}^{2}}}\right)^{2} & = & \frac{\lambda_{n}^{2}\bar{z}_{n}^{2}}{\sum_{i=1}^{n}\lambda_{i}^{2}\bar{z}_{n}^{2}}\\
 & \geq & \frac{-\lambda_{n}\sum_{i=1}^{n-1}\lambda_{i}\bar{z}_{i}^{2}}{\sum_{i=1}^{n-1}\lambda_{i}^{2}\bar{z}_{n}^{2}-\lambda_{n}\sum_{i=1}^{n-1}\lambda_{i}\bar{z}_{i}^{2}}\\
 & \geq & \frac{-\lambda_{n}\sum_{i=1}^{n-1}\lambda_{i}\bar{z}_{i}^{2}}{\max(\lambda_{1},-\lambda_{n})\sum_{i=1}^{n-1}\lambda_{i}\bar{z}_{n}^{2}-\lambda_{n}\sum_{i=1}^{n-1}\lambda_{i}\bar{z}_{i}^{2}}\\
 & = & \frac{-\lambda_{n}}{\max(\lambda_{1},-\lambda_{n})-\lambda_{n}}.
\end{eqnarray*}
The rest of the claim is straightforward.

(2) Through calculations we have seen in the proof of Lemma \ref{lem:control-z(x)},
we have 
\[
\bar{g}_{l,v}(\bar{z})=2\left|\frac{\sqrt{[v^{T}H\bar{z}]^{2}-[v^{T}Hv][\bar{z}^{T}H\bar{z}-2l]}}{v^{T}Hv}\right|=2\left|\frac{v^{T}H\bar{z}}{v^{T}Hv}\right|,
\]
since $l=f(\bar{z})=\frac{1}{2}\bar{z}^{T}H\bar{z}$. Now, 
\begin{eqnarray*}
|v^{T}Hv| & \geq & |\bar{v}^{T}H\bar{v}|-2|(v-\bar{v})^{T}H\bar{v}|-|(v-\bar{v})^{T}H(v-\bar{v})|\\
 & \geq & |\lambda_{n}|-2\alpha|\lambda_{n}|-\alpha^{2}\max(\lambda_{1},-\lambda_{n}).
\end{eqnarray*}
Finally, 
\begin{eqnarray*}
\frac{\bar{g}_{l,v}(\bar{z})}{\|\nabla\bar{f}(\bar{z})\|} & = & \frac{2|v^{T}H\bar{z}|}{\|H\bar{z}\||v^{T}Hv|}\\
 & \leq & \frac{2\|v\|\|H\bar{z}\|}{\|H\bar{z}\|[|\lambda_{n}|-2\alpha|\lambda_{n}|-\alpha^{2}\max(\lambda_{1},-\lambda_{n})]}\\
 & = & \frac{2}{|\lambda_{n}|-2\alpha|\lambda_{n}|-\alpha^{2}\max(\lambda_{1},-\lambda_{n})}.
\end{eqnarray*}

\end{proof}
We will use the following result.
\begin{prop}
\label{pro:product-norms}(Products and norms) Let $A_{i}$ and $\bar{A}_{i}$,
where $i=1,\dots,k$, be matrices such that the products $A_{1}A_{2}\cdots A_{k}$
and $\bar{A}_{1}\bar{A}_{2}\cdots\bar{A}_{k}$ are valid. Then 
\begin{equation}
\|A_{1}A_{2}\cdots A_{k}-\bar{A}_{1}\bar{A}_{2}\cdots\bar{A}_{k}\|\leq\left[\prod_{i=1}^{k}(\|\bar{A}_{i}\|+\|A_{i}-\bar{A}_{i}\|)\right]-\prod_{i=1}^{k}\|\bar{A}_{i}\|.\label{eq:product-norms}
\end{equation}
\end{prop}
\begin{proof}
The formula follows readily from 
\[
A_{1}A_{2}\cdots A_{k}-\bar{A}_{1}\bar{A}_{2}\cdots\bar{A}_{k}=[\bar{A}_{1}+(A_{1}-\bar{A}_{1})]\cdots[\bar{A}_{k}+(A_{k}-\bar{A}_{k})]-\bar{A}_{1}\bar{A}_{2}\cdots\bar{A}_{k}.
\]
\end{proof}
\begin{lem}
\label{lem:unif-diff}(Uniform bounds on differences) Suppose $f:\mathbb{R}^{n}\to\mathbb{R}$
satisfies Assumption \ref{ass:smooth-f}. Assume that $v$ is a unit
vector such that $\|v-\bar{v}\|<\alpha$. For every $\epsilon>0$,
there exists $\delta>0$, $\gamma>0$ and convex neighborhoods $U_{\delta}$
and $U_{\delta}^{\prime}$ of $0$ such that $P(f,\delta,\gamma,U_{\delta},U_{\delta}^{\prime})$
holds, and for all $x\in U_{\delta}$, we have 
\begin{enumerate}
\item $\|u(\nabla\bar{f}(\bar{z}(x)))-u(\nabla f(z(x)))\|<\epsilon$.
\item $\left|\frac{1}{u(\nabla\bar{f}(\bar{z}(x)))^{T}v}-\frac{1}{u(\nabla f(z(x)))^{T}v}\right|<\epsilon$.
\item $\left|\frac{\bar{g}_{l,v}(\bar{z})}{\|\nabla\bar{f}(\bar{z}(x))\|}-\frac{g_{l,v}(z)}{\|\nabla f(z(x))\|}\right|<\epsilon$,
where $f\big(\bar{z}(x)\big)=l$ and $l\in(-\gamma,0]$.
\end{enumerate}
\end{lem}
\begin{proof}
We use Lemma \ref{lem:control-z(x)}, which says that for $\epsilon_{1}>0$,
there are $\delta>0$, $\gamma>0$ and convex neighborhoods $U_{\delta}$
and $U_{\delta}^{\prime}$ of $0$ such that $P(f,\delta,\gamma,U_{\delta},U_{\delta}^{\prime})$
holds, and $\|\bar{z}(x)-z(x)\|\leq\epsilon_{1}\|\bar{z}(x)\|$ for
all $x\in U_{\delta}$. 

(1) We can easily obtain $\|z(x)\|\leq(1+\epsilon_{1})\|\bar{z}(x)\|$.
Let $H=\nabla^{2}f(0)$. Now,
\begin{eqnarray*}
\|\nabla\bar{f}(\bar{z}(x))-\nabla f(z(x))\| & \leq & \|\nabla\bar{f}(\bar{z}(x))-\nabla\bar{f}(z(x))\|+\|\nabla\bar{f}(z(x))-\nabla f(z(x))\|\\
 & \leq & \|H(\bar{z}(x)-z(x))\|+\left\Vert Hz(x)-\int_{0}^{1}\nabla^{2}f(tz(x))dt\cdot z(x)\right\Vert \\
 & \leq & \epsilon_{1}\|H\|\|\bar{z}(x)\|+\|z(x)\|\left\Vert H-\int_{0}^{1}\nabla^{2}f(tz(x))dt\right\Vert \\
 & \leq & \epsilon_{1}\|\bar{z}(x)\|\|H\|+\delta(1+\epsilon_{1})\|\bar{z}\|\\
 & = & \|\bar{z}(x)\|[\epsilon_{1}\|H\|+\delta(1+\epsilon_{1})].
\end{eqnarray*}
Note that the term $[\epsilon_{1}\|H\|+\delta(1+\epsilon_{1})]$ can
be made arbitrarily small. Note that $\frac{\|\nabla\bar{f}(\bar{z}(x))\|}{\|\bar{z}(x)\|}\geq\min(|\lambda_{n-1}|,|\lambda_{n}|)$,
so 
\begin{equation}
\frac{\|\nabla\bar{f}(\bar{z}(x))-\nabla f(z(x))\|}{\|\nabla\bar{f}(\bar{z}(x))\|}\leq\frac{[\epsilon_{1}\|H\|+\delta(1+\epsilon_{1})]}{\min(|\lambda_{n-1}|,|\lambda_{n}|)}.\label{eq:4.8-1}
\end{equation}
Next, for any $w_{1},w_{2}\in\mathbb{R}^{n}\backslash\{0\}$, we have
\begin{eqnarray}
\left\Vert \frac{w_{1}}{\|w_{1}\|}-\frac{w_{2}}{\|w_{2}\|}\right\Vert  & \leq & \left\Vert \frac{w_{1}}{\|w_{1}\|}-\frac{w_{2}}{\|w_{1}\|}\right\Vert +\left\Vert \frac{w_{2}}{\|w_{1}\|}-\frac{w_{2}}{\|w_{2}\|}\right\Vert \nonumber \\
 & \leq & \frac{\|w_{1}-w_{2}\|}{\|w_{1}\|}+\frac{\|w_{2}\|\left|\|w_{2}\|-\|w_{1}\|\right|}{\|w_{1}\|\|w_{2}\|}\nonumber \\
 & \leq & 2\frac{\|w_{1}-w_{2}\|}{\|w_{1}\|}.\label{eq:w1-w2}
\end{eqnarray}
Apply the observation in \eqref{eq:w1-w2} to \eqref{eq:4.8-1} to
get what we need.

(2) Let $M=\left[\sqrt{\frac{-\lambda_{n}}{\max(\lambda_{1},-\lambda_{n})-\lambda_{n}}}-\alpha\right]^{-1}$.
We have from Lemma \ref{lem:unif-terms}(1) that $\frac{1}{|u(\nabla\bar{f}(\bar{z}(x)))^{T}v|}\leq M$.
From (1), for $\epsilon_{1}>0$, there exist $\delta>0$, $\gamma>0$
and convex neighborhoods $U_{\delta}$ and $U_{\delta}^{\prime}$
of $0$ such that $P(f,\delta,\gamma,U_{\delta},U_{\delta}^{\prime})$
holds, and $|u(\nabla\bar{f}(\bar{z}(x)))^{T}v-u(\nabla f(z(x)))^{T}v|<\epsilon_{1}$
for all $x\in U_{\delta}$. First,
\begin{eqnarray*}
|u(\nabla f(z(x)))^{T}v| & \geq & |u(\nabla\bar{f}(\bar{z}(x)))^{T}v|-|u(\nabla\bar{f}(\bar{z}(x)))^{T}v-u(\nabla f(z(x)))^{T}v|\\
 & \geq & \frac{1}{M}-\epsilon_{1}\\
\Rightarrow\frac{1}{|u(\nabla f(z(x)))^{T}v|} & \leq & \frac{M}{1-\epsilon_{1}M}.
\end{eqnarray*}
Next, 
\begin{eqnarray*}
\left|\frac{1}{u(\nabla\bar{f}(\bar{z}(x)))^{T}v}-\frac{1}{u(\nabla f(z(x)))^{T}v}\right| & = & \frac{|v^{T}[u(\nabla\bar{f}(\bar{z}(x)))-u(\nabla f(z(x)))]|}{|u(\nabla\bar{f}(\bar{z}(x)))^{T}v||u(\nabla f(z(x)))^{T}v|}\\
 & \leq & M\frac{M}{1-\epsilon_{1}M}\|u(\nabla\bar{f}(\bar{z}(x)))-u(\nabla f(z(x)))\|\\
 & \leq & \frac{M^{2}}{1-\epsilon_{1}M}\epsilon_{1}.
\end{eqnarray*}
As the RHS converges to zero as $\epsilon_{1}\searrow0$, we are done.

(3) We use Proposition \ref{pro:product-norms} to get 
\begin{eqnarray*}
 &  & \left|\frac{\bar{g}_{l,v}(\bar{z})}{\|\nabla\bar{f}(\bar{z}(x))\|}-\frac{g_{l,v}(z)}{\|\nabla f(z(x))\|}\right|\\
 & \leq & [|\bar{g}_{l,v}(\bar{z})-g_{l,v}(z)|+|\bar{g}_{l,v}(\bar{z})|]\left[\left|\frac{1}{\|\nabla\bar{f}(\bar{z}(x))\|}\right|+\left|\frac{1}{\|\nabla\bar{f}(\bar{z}(x))\|}-\frac{1}{\|\nabla f(z(x))\|}\right|\right]\\
 &  & -\frac{|\bar{g}_{l,v}(\bar{z})|}{\|\nabla\bar{f}(\bar{z}(x))\|}\\
 & = & \left[\frac{|\bar{g}_{l,v}(\bar{z})-g_{l,v}(z)|}{\|\nabla\bar{f}(\bar{z}(x))\|}+\frac{|\bar{g}_{l,v}(\bar{z})|}{\|\nabla\bar{f}(\bar{z}(x))\|}\right]\left[1+\left|1-\frac{\|\nabla\bar{f}(\bar{z}(x))\|}{\|\nabla f(z(x))\|}\right|\right]-\frac{|\bar{g}_{l,v}(\bar{z})|}{\|\nabla\bar{f}(\bar{z}(x))\|}.
\end{eqnarray*}
By Lemma \ref{lem:control-z(x)}, for any $\epsilon_{1}>0$, there
exist $\delta>0$, $\gamma>0$ and convex neighborhoods $U_{\delta}$
and $U_{\delta}^{\prime}$ of $0$ such that $P(f,\delta,\gamma,U_{\delta},U_{\delta}^{\prime})$
holds, and $|\bar{g}_{l,v}(\bar{z})-g_{l,v}(z)|\leq\epsilon_{1}\|\bar{z}(x)\|$
for all $x\in U_{\delta}$. So 
\begin{eqnarray*}
\frac{|\bar{g}_{l,v}(\bar{z})-g_{l,v}(z)|}{\|\nabla\bar{f}(\bar{z}(x))\|} & \leq & \frac{\epsilon_{1}\|\bar{z}(x)\|}{\min(\lambda_{n-1},|\lambda_{n}|)\|\bar{z}(x)\|}\\
 & \leq & \frac{\epsilon_{1}}{\min(\lambda_{n-1},|\lambda_{n}|)}.
\end{eqnarray*}
Next, we may reduce $\delta$ if necessary so that $\|\nabla f(z(x))-\nabla\bar{f}(\bar{z}(x))\|\leq\epsilon_{1}\|\bar{z}(x)\|$for
all $x\in U_{\delta}$ as well. We have 
\begin{eqnarray*}
\left|1-\frac{\|\nabla\bar{f}(\bar{z}(x))\|}{\|\nabla f(z(x))\|}\right| & = & \frac{\left|\|\nabla f(z(x))\|-\|\nabla\bar{f}(\bar{z}(x))\|\right|}{\|\nabla f(z(x))\|}\\
 & \leq & \frac{\|\nabla f(z(x))-\nabla\bar{f}(\bar{z}(x))\|}{\|\nabla f(z(x))\|}\\
 & \leq & \frac{\|\nabla f(z(x))-\nabla\bar{f}(\bar{z}(x))\|}{\|\nabla\bar{f}(\bar{z}(x))\|-\|\nabla\bar{f}(\bar{z}(x))-\nabla f(z(x))\|}\\
 & \leq & \frac{\epsilon_{1}\|\bar{z}(x)\|}{\min(|\lambda_{n}|,|\lambda_{n-1}|)\|\bar{z}(x)\|-\epsilon_{1}\|\bar{z}(x)\|}\\
 & = & \frac{\epsilon_{1}}{\min(|\lambda_{n}|,|\lambda_{n-1}|)-\epsilon_{1}}.
\end{eqnarray*}
The RHS of both previous formulas converge to zero as $\epsilon_{1}\searrow0$,
and $\frac{|\bar{g}_{l,v}(\bar{z})|}{\|\nabla\bar{f}(\bar{z}(x))\|}$
is uniformly bounded by Lemma \ref{lem:unif-terms}(2). We have (3)
as needed.
\end{proof}
We have the following theorem.
\begin{thm}
\label{thm:convexity-smooth}(Hessian behavior and convexity) Let
$f:\mathbb{R}^{n}\to\mathbb{R}$ be such that $f$ is $\mathcal{C}^{2}$
in a neighborhood of a nondegenerate critical point $\bar{x}$ with
Morse index one, and let $\bar{v}$ be the eigenvector corresponding
to the negative eigenvalue of $\nabla^{2}f(\bar{x})$. Let $v$ be
a unit vector such that $\|v-\bar{v}\|=\alpha$ is small. Then for
any $\epsilon>0$, there are $\delta>0$, $\gamma>0$ and convex neighborhoods
$U_{\delta}$ and $U_{\delta}^{\prime}$ satisfying $P(f,\delta,\gamma,U_{\delta},U_{\delta}^{\prime})$
such that 
\begin{itemize}
\item For all $l\in(f(\bar{x})-\gamma,f(\bar{x})]$, $g_{l,v}^{2}:U_{\delta}\to\mathbb{R}$
is convex on $U_{\delta}$ 
\item $\|\nabla^{2}g_{l,v}^{2}(x)-\nabla^{2}\bar{g}_{l,v}^{2}(x)\|\leq\epsilon$
for all $x\in U_{\delta}$ satisfying $g_{l,v}(x)>0$. Here $\nabla^{2}\bar{g}_{l,v}^{2}(x)$
equals \textup{$\frac{8}{(v^{T}Hv)^{2}}[Hvv^{T}H-(v^{T}Hv)H]$, where
$H=\nabla^{2}f(\bar{x})$ by Proposition \ref{pro:quad-formula}. }
\end{itemize}
\end{thm}
\begin{proof}
The formulas for the Hessian $\nabla^{2}\bar{g}_{l,v}(x)$ in Proposition
\ref{pro:Grad-Hess-g-2} and Proposition \ref{pro:quad-formula} are
equal. We want to show that for all $\epsilon>0$, there exists $\delta>0$
and a neighborhood $U_{\delta}$ of $x$ such that $\|\nabla^{2}\bar{g}_{l,v}^{2}(x)-\nabla^{2}g_{l,v}^{2}(x)\|<\epsilon$
for all $x\in U_{\delta}$. Without loss of generality, suppose Assumption
\ref{ass:smooth-f} holds. The formulas for $\nabla^{2}g_{l,v}^{2}(x)$
and $\nabla^{2}\bar{g}_{l,v}^{2}(x)$ in Proposition \ref{pro:Grad-Hess-g-2}
can be written as 
\begin{eqnarray*}
\nabla^{2}(g_{l,v}^{2})(x) & = & 2\left(-\frac{u(\nabla f(z))}{u(\nabla f(z))^{T}v}+\frac{u(\nabla f(z^{\prime}))}{u(\nabla f(z^{\prime}))^{T}v}\right)\left(-\frac{u(\nabla f(z))}{u(\nabla f(z))^{T}v}+\frac{u(\nabla f(z^{\prime}))}{u(\nabla f(z^{\prime}))^{T}v}\right)^{T}\\
 &  & -2\frac{g_{l,v}(x)}{\|\nabla f(z)\|}\left(I-\frac{u(\nabla f(z))v^{T}}{v^{T}u(\nabla f(z))}\right)\frac{\nabla^{2}f(z)}{v^{T}u(\nabla f(z))}\left(I-\frac{u(\nabla f(z))v^{T}}{v^{T}u(\nabla f(z))}\right)^{T}\\
 &  & +2\frac{g_{l,v}(x)}{\|\nabla f(z^{\prime})\|}\left(I-\frac{u(\nabla f(z^{\prime}))v^{T}}{v^{T}u(\nabla f(z^{\prime}))}\right)\frac{\nabla^{2}f(z^{\prime})}{v^{T}u(\nabla f(z^{\prime}))}\left(I-\frac{u(\nabla f(z^{\prime}))v^{T}}{v^{T}u(\nabla f(z^{\prime}))}\right)^{T},
\end{eqnarray*}
where $u(x)=x/\|x\|$. These formulas can be rewritten as finite sums
of products of terms of the form $\frac{1}{v^{T}u(\nabla f(z(x)))}$,
$u(\nabla f(z))$, $\nabla^{2}f(z)$, $\frac{g_{l,v}(x)}{\|\nabla f(z)\|}$
and other terms involving $z^{\prime}$. We can establish the positive
definiteness of the $\nabla^{2}g_{l,v}^{2}(x)$ for $x$ close to
$\bar{x}$ by ensuring that $\|\nabla^{2}g_{l,v}^{2}(x)-\nabla^{2}\bar{g}_{l,v}^{2}(x)\|$
goes to zero. This is immediate from Proposition \ref{pro:product-norms}
and Lemmas \ref{lem:unif-terms} and \ref{lem:unif-diff}.

Applying Proposition \ref{pro:convex-from-pd} gives us the result
in hand.\end{proof}
\begin{rem}
(The case $l>f(\bar{x})$) A result similar to Theorem \ref{thm:convexity-smooth}
establishing the convexity of $g_{l,v}^{2}(\cdot)$ for $l>f(\bar{x})$
like in Proposition \ref{pro:quad-formula} would be attractive. But
for $l>f(\bar{x})$, the vectors $\bar{z}(x)$ and $z(x)$ may not
exist at all. Yet another issue to consider is that $g_{l,v}(x)$
may be positive but $\bar{g}_{l,v}(x)$ is zero, or vice versa, making
comparison with $\nabla^{2}g_{l,v}^{2}$ and $\nabla^{2}\bar{g}_{l,v}^{2}$
more difficult. Even if these were not an issue, $\bar{v}^{T}\nabla\bar{f}(\bar{z}(x))$
could be zero or be close to zero for some $x$, resulting in a division
by zero in the formulas in Proposition \ref{pro:Grad-Hess-g-2}. Nevertheless,
the Hessian $\nabla^{2}g_{l,v}^{2}(\cdot)$ is still positive definite
if $v^{T}\nabla f(z(x))$ is sufficiently far from $0$. If it turns
out that the Hessian is positive definite whenever $g_{l,v}(x)>0$,
then one can still use Proposition \ref{pro:convex-from-pd} to establish
convexity. 
\end{rem}

\section{\label{sec:implement}Observations from a numerical experiment }

In this section, we implement a simple version of the mountain pass
algorithm on a two dimensional problem (called the six hump camel
back function in \cite{MM04}) defined by 
\begin{equation}
f(x_{1},x_{2})=(4-2.1x_{1}^{2}+x_{1}^{4}/3)x_{1}^{2}+x_{1}x_{2}+4(x_{2}^{2}-1)x_{2}^{2}.\label{eq:six-hump}
\end{equation}

In our numerical experiments, we only seek to obtain graphical information
from this two dimensional example that the parallel distance is a
good strategy. We calculate the Hessians $\nabla^{2}f(x)$ at each
evaluation. While practical implementations will not calculate the
Hessian, we can study the potential of methods that create second
order models from previous gradient evaluations.   

\begin{figure}
\includegraphics[scale=0.5]{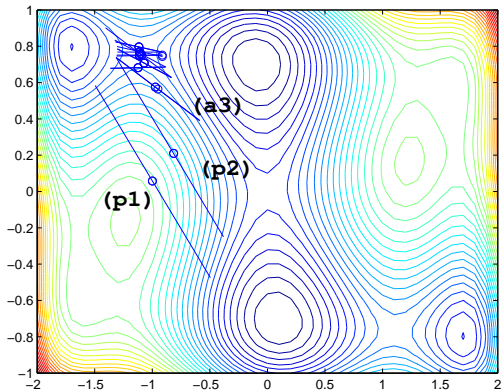}

\caption{\label{fig:sample-run}The first two iterates of Algorithm \ref{alg:new-alg}
are operations (PD) to reduce the parallel distance, marked as (p1)
and (p2), and the third iterate, marked as (a3), is an (Av) operation
to adjust the vector $v$. In this case, the algorithm concentrates
its efforts close to the saddle point near $(-1,0.8)$. If a different
endpoint had been fixed in (a3) instead, the algorithm might have
found the saddle point $(0,0)$. The saddle point $(0,0)$ has a higher
value, and would be of greater interest as the bottleneck.}
\end{figure}

We look at Figure \ref{fig:sample-run}. One observation that can
be made for Algorithm \ref{alg:new-alg} is that while Algorithm \ref{alg:new-alg}
focuses its computations on a saddle point in two runs of (PD) and
one run of (Av), it did not focus its computations on the saddle point
$(0,0)$, which has a higher critical value. We can see this phenomenon
as part of the risks involved in trying to zoom computations to a
saddle point. Moreover, this is unavoidable because in a general problem,
an optimal mountain pass may be difficult to find by any method. Furthermore,
for this example, when the mountain pass algorithm is run between
the saddle point near $(-1,0.8)$ and the local minimizer near $(0.1,-0.7)$,
it may find the saddle point $(0,0)$.

\section{\label{sec:Conclusions}Conclusion}

We propose two Principles (P1) and (P2) that a good mountain pass
algorithm should satisfy. We proposed the subroutine (PD) in Algorithm
\ref{alg:subroutines} to build our global mountain pass algorithm
in Algorithm \ref{alg:new-alg}, making use of the parallel distance
$g_{l,v}(\cdot)$. Through Proposition \ref{pro:quad-formula}, we
see that $g_{l,v}(\cdot)^{2}$ satisfies (P1$^{\prime}$), and that
(P2) follows from work in \cite{mountain}. Sections \ref{sec:Basic-parallel}
and \ref{sec:Convexity-para-dist} discuss how $g_{l,v}(\cdot)^{2}$
satisfies property (P1).

Finally, we envision that a robust mountain pass algorithm should
include quadratic model methods, level set methods and path-based
methods. For example, the points chosen for function and gradient
evaluations in a level set method should be such that they provide
insight for quadratic model methods and path-based methods. The right
blend of these methods allow them to overcome each other's shortcomings.
The evidence from our numerical experiments so far are encouraging.
\begin{acknowledgement*}
 We thank Jiahao Chen for discussions that led to the main ideas
in this paper and for references to the literature on computing saddle
points, and to James Renegar for some helpful discussions.
\end{acknowledgement*}
\bibliographystyle{alpha}
\bibliography{../refs}

\end{document}